\documentclass[10pt]{amsart}
\usepackage{latexsym}
\usepackage{amsmath}
\usepackage{amssymb}
\usepackage{mathrsfs}
\usepackage{graphicx}
\usepackage{color}
\usepackage{pgfpages}
\usepackage{ifthen}
\usepackage{leftidx,tensor}
\usepackage[T1]{fontenc}
\usepackage[latin1]{inputenc}
\usepackage{mathtools}
\usepackage{comment}
\usepackage{dsfont}
\usepackage[nocompress]{cite}
\usepackage{tikz}
\usepackage{subcaption}

\usepackage[shortlabels]{enumitem}
\usepackage{aliascnt}
\usepackage[bookmarks=true,pdfstartview=FitH, pdfborder={0 0 0}, colorlinks=true,citecolor=red, linkcolor=blue]{hyperref}
\usepackage{bbm}
\usepackage{nicefrac}

\definecolor{solidgreen}{RGB}{34, 139, 34}
\definecolor{fluidblue}{RGB}{0, 102, 204}
\definecolor{interfaceorange}{RGB}{255, 140, 0}
\definecolor{refgray}{RGB}{120, 120, 120}

\theoremstyle{plain}
\newtheorem{thm}{Theorem}[section]
\newaliascnt{cor}{thm}
\newaliascnt{prop}{thm}
\newaliascnt{lem}{thm}
\newtheorem{cor}[cor]{Corollary}
\newtheorem{prop}[prop]{Proposition}
\newtheorem{lem}[lem]{Lemma}
\aliascntresetthe{cor}
\aliascntresetthe{prop}
\aliascntresetthe{lem} 
\theoremstyle{definition}
\newaliascnt{defn}{thm}
\newaliascnt{asu}{thm}
\newaliascnt{con}{thm}
\newtheorem{defn}[defn]{Definition}

\aliascntresetthe{defn}
\aliascntresetthe{asu}
\aliascntresetthe{con}
\newcounter{stp}
\newcounter{stpi}
\newcounter{stpci}
\newcounter{stpiii}

\theoremstyle{remark}
\newaliascnt{rem}{thm}
\newaliascnt{exa}{thm}
\newaliascnt{masu}{thm}
\newaliascnt{nota}{thm}
\newaliascnt{sett}{thm}
\newtheorem{rem}[rem]{Remark}

\aliascntresetthe{rem}
\aliascntresetthe{exa}
\aliascntresetthe{masu}
\aliascntresetthe{nota}
\aliascntresetthe{sett}
\numberwithin{equation}{section}

\setlist[enumerate]{font = \normalfont}

\newcommand{\F}{\mathbb{F}}
\newcommand{\E}{\mathbb{E}}
\newcommand{\R}{\mathbb{R}}
\newcommand{\C}{\mathbb{C}}
\newcommand{\B}{\mathbb{B}}
\newcommand{\N}{\mathbb{N}}
\newcommand{\bP}{\mathbb{P}}

\newcommand{\D}{\mathbb{D}}

\newcommand{\oB}{\overline{\B}}
\newcommand{\tE}{\Tilde{\mathbb{E}}}

\newcommand{\rW}{\mathrm{W}} 
\newcommand{\rL}{\mathrm{L}}

\newcommand{\rG}{\mathrm{G}}
\newcommand{\rH}{\mathrm{H}}

\newcommand{\rC}{\mathrm{C}}
\newcommand{\rB}{\mathrm{B}}

\newcommand{\rD}{\mathrm{D}}

\newcommand{\rX}{\mathrm{X}}
\newcommand{\rY}{\mathrm{Y}}

\newcommand{\BUC}{\mathrm{BUC}}

\newcommand{\rd}{\mathrm{d}}
\newcommand{\srd}{\,\mathrm{d}}

\newcommand{\fs}{\mathrm{fs}}
\newcommand{\mfs}{\mathrm{mfs}}
\newcommand{\fl}{\mathrm{fl}}
\newcommand{\pr}{\mathrm{pr}}

\newcommand{\mre}{\mathrm{e}}
\newcommand{\mri}{\mathrm{i}}

\newcommand{\ess}{\mathrm{ess}}

\newcommand{\rf}{\mathrm{f}}
\newcommand{\rs}{\mathrm{s}}

\newcommand{\rp}{\mathrm{p}}
\renewcommand{\rm}{\mathrm{m}}

\newcommand{\cS}{\mathcal{S}}
\newcommand{\cF}{\mathcal{F}}

\newcommand{\cL}{\mathcal{L}}

\newcommand{\cK}{\mathcal{K}}
\newcommand{\cJ}{\mathcal{J}}
\newcommand{\cT}{\mathcal{T}}
\newcommand{\cR}{\mathcal{R}}

\newcommand{\cH}{\mathcal{H}}
\newcommand{\Hinfty}{\cH^\infty}
\newcommand{\cBIP}{\mathcal{BIP}}

\newcommand{\td}{\Tilde{d}}

\newcommand{\tA}{\Tilde{A}}
\newcommand{\tB}{\Tilde{B}}
\newcommand{\tC}{\Tilde{C}}

\newcommand{\eps}{\varepsilon}
\newcommand{\del}{\partial}

\DeclareMathOperator{\diag}{diag}
\DeclareMathOperator{\Id}{Id}
\DeclareMathOperator{\mdiv}{div}
\DeclareMathOperator{\Rep}{Re}

\DeclareMathOperator*{\Cof}{Cof}
\newcommand{\tin}{\enspace \text{in} \enspace}
\newcommand{\ton}{\enspace \text{on} \enspace}
\newcommand{\tfor}{\enspace \text{for} \enspace}
\newcommand{\tforall}{\enspace \text{for all} \enspace}
\newcommand{\tand}{\enspace \text{and} \enspace}

\newcommand{\tif}{\enspace \text{if} \enspace}
\newcommand{\twhere}{\enspace \text{where} \enspace}
\newcommand{\twith}{\enspace \text{with} \enspace}

\begin{document}

\title[Strong time-periodic solutions to a multilayered FSI problem]{Strong time-periodic solutions for a multilayered fluid-structure interaction system with nonlinear coupling}

\author{Felix Brandt*}
\thanks{*Corresponding Author}
\address{Department of Mathematics, University of California at Berkeley, Berkeley, 94720, CA, USA.}
\email{fbrandt@berkeley.edu}
\author{Claudiu M\^{i}ndril\u{a}}
\address{Basque Center for Applied Mathematics (BCAM), Alameda de Mazarredo 14, 48009 Bilbao, Spain.}
\email{cmindrila@bcamath.org}
\author{Arnab Roy}
\address{Basque Center for Applied Mathematics (BCAM), Alameda de Mazarredo 14, 48009 Bilbao, Spain.}
\address{IKERBASQUE, Basque Foundation for Science, Plaza Euskadi 5, 48009 Bilbao, Bizkaia, Spain.}
\email{aroy@bcamath.org}
\subjclass[2020]{74F10, 35R35, 35B10, 35Q30, 74H20, 35K59}
\keywords{Multilayered fluid-structure interaction problem, strong time-periodic solutions, nonlinear coupling, damped wave problem, multilayered fluid-structure operator, decoupling approach, Arendt-Bu theorem}

\begin{abstract}
We investigate a time-periodic fully three-dimensional fluid-structure interaction system in which the Navier-Stokes equations for an incompressible viscous fluid are coupled with a multilayered elastic structure composed of a damped thin linear plate and a thick viscoelastic layer. 
The coupling is nonlinear, meaning that it is on a moving interface that is not known a priori, rendering the problem a moving-domain problem.
We prove the existence of strong time-periodic solutions. 
The proof relies on a fixed point argument, combining sharp nonlinear estimates with a detailed analysis of the linearized system. The linearized problem is analyzed by employing the Arendt-Bu theorem on maximal periodic $\rL^p$-regularity, which requires several new analytical ingredients including a refined lifting procedure, a decoupling strategy establishing $\mathcal{R}$-sectoriality of the coupled operator, a careful treatment of the thick structural layer, and a spectral analysis adapted to the multilayered setting. This provides \textit{the first strong time-periodic existence result for multilayered fluid-structure interaction systems}, and the methods are expended to extend more broadly to nonlinear coupled PDEs on moving domains with periodic forcing.
\end{abstract}

\maketitle

\section{Introduction}\label{sec:intro}

Fluid-structure interaction (FSI) problems with dynamic interfaces have become a prominent subject of research, driven by both theoretical challenges and practical applications. 
A natural classification arises based on the position of the structure relative to the fluid: 
Some models describe structures immersed in the fluid, while others consider structures attached to the boundary of the fluid domain.
Thorough reviews of analytical developments and practical applications in FSI theory can be found in the excellent surveys by Kaltenbacher et al.~\cite{KKLTTW:18} and \v{C}ani\'c \cite{Can:21} and the references cited therein.

The present paper establishes for the first time the existence of a strong time-periodic solution to a fully three-dimensional {\em multilayered} FSI problem subject to time-periodic forcing terms.
In fact, we consider the nonlinearly coupled interaction problem of a three-dimensional viscous, incompressible Newtonian fluid with a multilayered structure consisting of a two-dimensional thin plate-type and a three-dimensional thick (visco-)elastic layer.
Here the term {\em nonlinear coupling} means that the fluid-structure interface is not fixed, but it is evolving in time, giving rise to a moving boundary problem.
The model will be introduced in detail in \autoref{sec:multilayered FSI probl & main result} and can be found in \eqref{eq:multilayered fsi}.
The motivation for studying multilayered FSI problems is strongly rooted in biomedical applications, where many physiological structures, such as blood vessels, are composed of multiple elastic and viscoelastic layers.
We refer here for example to plaque growth in arteries, which is a key factor in atherosclerosis and whose modeling and numerical simulations can, e.g., be found in \cite{Gor:17, PRVVZ:21}.
Another application of multilayered FSI problems concerns bioartificial organ design, see for instance \cite{WCBBR:22}.
 
Previous studies in the context of boundary coupled FSI problems primarily focused on the analysis of the initial value (Cauchy) problem, and are mainly restricted to the situation of a single, thin structure attached to the fluid boundary.
First results on the existence of weak solutions to incompressible fluids coupled with elastic plates are due to Chambolle et al.~\cite{CDEG:05} and Grandmont \cite{Gra:08}.
The analysis was extended to models involving linear elastic \cite{LR14} and nonlinear \cite{MS:22} Koiter shells.
A constructive approach to the existence of weak solutions to incompressible fluid-linear elastic shell systems via the Arbitrary Lagrangian-Eulerian (ALE) method was developed by Muha and \v{C}ani\'c \cite{MC13}.
More recently, the same authors extended in \cite{MC:14} the scope of FSI models to incorporate multilayered structures, featuring the coupling of a thin shell with an additional thick elastic layer, and laying the groundwork for more realistic biomechanical and engineering applications.
The model from \cite{MC:14} is the basis for the present work.

Concerning {\em strong} solutions to FSI problems with elastic structures located at the fluid boundary, a first local-in-time well-posedness result for fluid-structure systems involving viscoelastic beams and small initial data was established in the foundational work \cite{BdV:04}.
Remarkably, Grandmont and Hillairet \cite{GH:16} proved the global-in-time strong well-posedness of such FSI problems in a 2D fluid-1D beam setting for large data. 
Strong solutions to FSI problems involving nonlinear elastic Koiter shells were studied in~\cite{CS:10, MRR:20}, while approaches to FSI problems of incompressible fluids and damped (thin) structures via $\rL^p$-theory have been carried out by Denk and Saal \cite{DS:20} and Maity and Takahashi \cite{MT:21}.
Let us emphasize that all of these works concern the initial value problem, and none of these works addresses the situation of a multilayered structure. 

There are contributions on multilayered and coupled PDE systems, including studies on multilayered Lam\'{e}-heat systems or heat-wave-wave interactions, which provide insight into stability, decay, and spectral properties of multilayered coupled systems in linear settings \cite{ALT:16, AGM:20, AG:22, AGM:22}.
Multilayered FSI problems involving poroelastic structures were studied in the linearly coupled case in \cite{BCMW:21} and in the nonlinearly coupled case in \cite{KCM:24, BCM:25}.
Recently, attention has also turned to FSI problems of greater complexity, including the case of compressible or heat-conducting fluids coupled with beams and shells, see, e.g., \cite{BS18, MRT:21, MMNRT:22}.

The study of {\em time-periodic} problems is a classical topic in mathematical fluid mechanics.
We refer here for instance to the pioneering articles by Serrin \cite{Ser:59}, Judovi\v{c} \cite{Jud:60} and Prodi \cite{Pro:60} as well as the works of Kozono and Nakao \cite{KN:96}, Galdi and Sohr \cite{GS:04} or Geissert, Hieber and Nguyen \cite{GHN:16}. 
On the other hand, time-periodicity plays a central role in modeling phenomena where cyclic behavior is inherent, especially in biomedical applications.
For example, the rhythmic nature of blood flow, governed by the heart's repetitive pumping, naturally leads to models with periodic forcing.
However, despite their relevance, FSI problems in the time-periodic setting remain underexplored.
Several recent studies have taken important steps toward overcoming the associated analytical difficulties.
In the case of incompressible fluids interacting with (single) elastic structures, results on strong solutions have been obtained by Casanova \cite{Casanova} as well as Celik and Kyed \cite{CK:21}.
The existence of weak solutions to FSI problems in time-periodic settings has been rigorously investigated in \cite{Claudiu22, Claudiu23} and \cite{MR25}, where the latter work addresses a multilayered FSI problem.
For time-periodic solutions to FSI problems involving different fluid models, such as compressible ones, and structural configurations, we also refer to \cite{JPP:20, PP:24, KMNT:25}.
Time-periodic weak solutions to parabolic-hyperbolic systems have also been investigated in \cite{MMSW:24}, while time-periodic fluid-{\em rigid body} interaction problems have been explored, e.g., by Bonheure and Galdi \cite{BG:25}.

In the following, we describe the strategy of the proof of the strong time-periodic result.
Due to the nonlinear coupling and the resulting time-dependent fluid domain, we employ a change of variables that acts locally in space to reduce the problem to the situation of a fixed fluid domain problem, but with terms that are non-local in time.
We also rewrite the main result in this reference configuration.

In order to tackle the full, nonlinear, time-periodic problem, we use a fixed point argument, which in turns requires a good understanding of the linearized problem combined with suitable nonlinear estimates.
Our approach relies on the {\em Arendt-Bu theorem} \cite{AB:02} based on operator-valued discrete Fourier multipliers.
Roughly speaking, the latter theorem characterizes the strong solvability of the linearized {\em time-periodic problem} in terms of maximal $\rL^p$-regularity of the {\em initial value problem} and a spectral condition on the semigroup, or, equivalently, its generator, see also \autoref{prop:Arendt-Bu thm}.
For further background on operator theory and maximal $\rL^p$-regularity, we also refer to the monographs \cite{Ama:95, DHP:03, PS:16} or the book chapter \cite{KW:04}.

With regard to the Arendt-Bu theorem, it is a major task to establish the maximal $\rL^p$-regularity of the initial value problem.
In fact, we first show the {\em boundedness of the $\Hinfty$-calculus} of the associated damped wave operator matrix in \autoref{prop:bdd Hinfty & spectral bound thick structural layer}, which especially implies its maximal $\rL^p$-regularity.

For handling the kinematic coupling conditions on equality of velocities and displacements as made precise below in \eqref{eq:cont of velocities} and \eqref{eq:cont of displacement}, we then use a lifting procedure, see \autoref{lem:sol stat Stokes probl with inhom bc} and \autoref{lem:solvability stat Lame probl}.
Inspired by \cite{MT:21} in the case of a single thin structure, see also the works \cite{MT:18, EMT:23, BHR:26} in the context of fluid-rigid body interaction problems, invoking the so-called {\em added mass operator}, and employing the above lifting procedure, we then reformulate the linearized problem in terms of a {\em multilayered fluid-structure operator}~$A_\mfs$ in \autoref{prop:equiv reform lin probl} and \eqref{eq:reform in op form multilayered fs op}.
The strongly coupled nature of the present multilayered FSI problem leads to a highly {\em non-diagonal domain}.
Thus, to establish the maximal $\rL^p$-regularity, or, equivalently, the $\cR$-sectoriality of $A_\mfs$ thanks to the celebrated result by Weis \cite{Wei:01}, we use a decoupling approach inspired by \cite{BBH:24b}.
More precisely, building upon the aforementioned lifting procedure, we invoke a similarity transform that leads to an operator with diagonal domain, but of a more complicated shape.
Together with a suitable decomposition of the decoupled operator matrix and perturbation arguments, this yields the maximal $\rL^p$-regularity of the multilayered fluid-structure operator $A_\mfs$ in \autoref{prop:cR-sect and max reg up to shift}. 

For the linear theory, it then remains to verify the invertibility of $A_\mfs$.
Due to the damping in the thick layer, the spectrum $\sigma(A_\mfs)$ does {\em not} only consist of eigenvalues, so it is insufficient to merely study the associated eigenvalue problem.
Instead, we use another decoupling argument and successively unveil the shape of the essential spectrum.
Joint with the control of the point spectrum, this leads to the invertibility of $A_\mfs$ in \autoref{prop:spectral theory multilayered fs op}, and culminates in the maximal periodic regularity, see \autoref{thm:max per reg multilayered fs op}.

Finally, we derive the nonlinear estimates, where a particular focal point is the handling of the terms emerging from the change of variables to the fixed domain.
Note that the nonlinear estimates required in the periodic setting differ from those typically used in initial value problems.
The fixed point argument is then completed upon combining the maximal periodic regularity with the nonlinear estimates.

Let us elaborate on the main contributions and methodological advances of this article.
To the best of our knowledge, this work appears to be the first to establish {\em the existence of strong solutions for a nonlinearly coupled} multilayered FSI problem, where the fluid is described by the Navier-Stokes equations. 
Moreover, it seems to be the first paper to address strong time-periodic solutions to a multilayered fluid-structure interaction problem.
We believe that by a similar methodology as in the present paper, the corresponding initial value problem could be tackled to establish its local strong well-posedness and global strong well-posedness for small data.
In fact, one would employ maximal $\rL^p$-regularity instead of maximal periodic $\rL^p$-regularity as in this manuscript.
This is left to future study.
The presence of the thick structural layer introduces substantial analytical challenges, requiring significant extensions beyond existing FSI theory. 
A key technical contribution is the bounded $\mathcal{H}^\infty$-calculus of the operator matrix associated with the damped wave system (see \autoref{prop:bdd Hinfty & spectral bound thick structural layer}), which, to the best of our knowledge, is new and may be of independent interest. The presence of the thick layer also necessitates a further step in the lifting procedure, and the resulting multilayered fluid-structure operator is considerably more complex than in the case of a single thin layer as in \cite{MT:21}. 
In particular, to tackle the strongly coupled nature of the present multilayered problem, a more intricate decoupling approach tailored to the functional analytic setting of the damped wave problem is required. As far as we are aware, this is the first time such a technique has been employed in the context of boundary-coupled FSI problems with elastic structures, though it may prove useful in a broader class of FSI models.
Furthermore, the damping in the thick structural layer renders the spectral analysis substantially more delicate. In particular, the determination of the essential spectrum requires a careful decomposition of the decoupled operator matrix. 

This article has the following outline.
In \autoref{sec:not and function spaces}, we settle some notation and make precise the function spaces used throughout this paper.
In \autoref{sec:multilayered FSI probl & main result}, we introduce the time-periodic multilayered FSI problem and present the main result on the existence of a strong time-periodic solution for small periodic forces.
\autoref{sec:trafo to fixed dom} is dedicated to the transformation to a fixed fluid domain.
The central \autoref{sec:analysis of the lin probl} discusses the linear theory for the time-periodic problem, including the damped wave problem, the reformulation in operator form, and the maximal $\rL^p$-regularity and invertibility of the resulting operator matrix.
In \autoref{sec:ex of strong time per sol}, we establish the nonlinear estimates and complete the proof of the main result.

\section{Notation and function spaces}\label{sec:not and function spaces}

In this section, we make precise some pieces of notation and function spaces.
In fact, the subscript $_\rf$ will be used to denote objects related to the fluids, while the subscript $_\rs$ indicates objects associated with the thick structural layer.
On the other hand, the superscript $^\cF$ will be employed for variables on the moving fluid domains.
Moreover, $\Omega$ will be used for three-dimensional domains, whereas $\Gamma$ will represent boundaries or interfaces.
Concerning differential operators, we use $\Delta_s$ and $\nabla_s$ to designate the horizontal Laplacian and gradient, respectively, i.e.,
\begin{equation*}
    \Delta_s = \del_1^2 + \del_2^2 \tand \nabla_s = \binom{\del_1}{\del_2}.
\end{equation*}
For Banach spaces $\rX$ and $\rY$, we denote by $\cL(\rX,\rY)$ the set of bounded linear operators from $\rX$ to $\rY$, with the convention $\cL(\rX) \coloneqq \cL(\rX,\rX)$.

Next, let us introduce the function spaces used throughout the paper.
For a sufficiently regular domain $G \subset \R^n$, $n \in \N$, as well as $m \in \N$ and $p \in [1,\infty]$, we use $\rL^p(G)$ and $\rW^{m,p}(G)$ to denote the Lebesgue and Sobolev spaces.
We will use the subscript $_\rm$ to denote functions with average zero in $\rL^p(G)$, i.e.,
\begin{equation*}
    \rL_\rm^p(G) \coloneqq \left\{f \in \rL^p(G) : \int_G f \srd x = 0\right\}.
\end{equation*}
Besides, $\rW_0^{1,p}(G)$ denotes the functions in $\rW^{1,p}(G)$ with zero trace on $\partial G$.
Moreover, for $s \ge 0$, $p \in (1,\infty)$ and $q \in [1,\infty]$, we denote by $\rH^{s,p}(G)$, $\rW^{s,p}(G)$ and $\rB_{pq}^s(G)$ the Bessel potential spaces, fractional Sobolev or Sobolev-Slobodeckij spaces and Besov spaces, respectively.
Note that the Bessel potential spaces and Besov spaces $\rH^{s,p}(\R^n)$ and $\rB_{pq}^s(\R^n)$ on the whole space can be defined in terms of the Bessel potential and dyadic partitions, respectively, while we set
\begin{equation*}
    \rW^{s,p}(\R^n) \coloneqq \left\{
    \begin{aligned}
        &\rH^{s,p}(\R^n) &&\tif s \in \N_0,\\
        &\rB_{pp}^s(\R^n) &&\tif s \notin \N_0.
    \end{aligned}
    \right.
\end{equation*}
The respective spaces on $G$ can then be obtained via restriction.
We do not provide more details here, but we refer to \cite{Tri:78} instead for further background.
Let us only observe that for $m \in \N$, $p \in (1,\infty)$, $q \in [1,\infty]$ and $\theta \in (0,1)$, we have the interpolation relations
\begin{equation*}
    [\rL^p(G),\rW^{m,p}(G)]_\theta = \rH^{\theta m,p}(G) \tand (\rL^p(G),\rW^{m,p}(G))_{\theta,q} = \rB_{pq}^{\theta m}(G),
\end{equation*}
where $[\cdot,\cdot]_\theta$ and $(\cdot,\cdot)_{\theta,q}$ represent the complex and real interpolation functor, respectively.
For $s \ge 0$ and $p \in (1,\infty)$, we will also use the notation $\rW_\rm^{s,p}(G) \coloneqq \rW^{s,p}(G) \cap \rL_\rm^p(G)$.
In the context of continuous functions, we use the subscript $_\mathrm{c}$ to indicate a compact support, while by $\mathrm{BUC}(G)$, we denote the bounded and uniformly continuous functions on $G$.

We conclude this section by introducing function spaces on moving domains, which will be required to rigorously define functions on the time-dependent fluid domain.
To this end, we consider the reference fluid domain $\Omega_\rf$ and the moving fluid domain $\Omega_\rf(t)$, which will be introduced in \autoref{sec:multilayered FSI probl & main result}. Moreover, we consider a $\rC^1$-diffeomorphism $X$ mapping $\Omega_\rf$ onto $\Omega_\rf(t)$. This mapping will be constructed in \autoref{sec:trafo to fixed dom} and is assumed to satisfy
$X \in \rL^p\bigl(0,T;\rW^{2,q}(\Omega_\rf)^3\bigr)
\cap
\rW^{2,p}\bigl(0,T;\rL^q(\Omega_\rf)^3\bigr)$.
For $p,q \in (1,\infty)$, we then define the following function spaces
\begin{equation*}
    \begin{aligned}
        \rL^p(0,T;\rL^q(\Omega_\rf(\cdot)))
        &\coloneqq \left\{f \circ X^{-1} : f \in \rL^p(0,T;\rL^q(\Omega_\rf))\right\},\\
        \rL^p(0,T;\rW^{2,q}(\Omega_\rf(\cdot)))
        &\coloneqq \left\{f \circ X^{-1} : f \in \rL^p(0,T;\rW^{2,q}(\Omega_\rf))\right\},\\
        \rW^{1,p}(0,T;\rL^q(\Omega_\rf(\cdot)))
        &\coloneqq \left\{f \circ X^{-1} : f \in \rW^{1,p}(0,T;\rL^q(\Omega_\rf))\right\},\\
        \rC([0,T];\rW^{1,q}(\Omega_\rf(\cdot)))
        &\coloneqq \left\{f \circ X^{-1} : f \in \rC([0,T];\rW^{1,q}(\Omega_\rf))\right\},\\
        \rC([0,T];\rB_{qp}^{2-\nicefrac{2}{p}}(\Omega_\rf(\cdot)))
        &\coloneqq \left\{f \circ X^{-1} : f \in \rC([0,T];\rB_{qp}^{2-\nicefrac{2}{p}}(\Omega_\rf))\right\},\\
    \end{aligned}
\end{equation*}
where $(f \circ X^{-1})(t,x) \coloneqq f(t,(X(t,\cdot))^{-1}(x))$ for simplicity of notation.

\section{The multilayered FSI problem and main result}\label{sec:multilayered FSI probl & main result}

In this section, we introduce the multilayered FSI problem and then state the main result on the existence of a strong time-periodic solution to this problem under the action of small time-periodic forces.
First, we describe the underlying geometry.
Let $\Omega_\rf \subset \R^3$ be the reference fluid domain. Its boundary $\del \Omega_\rf$ contains a flat portion $\Gamma_0$ corresponding to the reference configuration of the plate, given by
$\Gamma_0 = \omega \times \{0\}$,
where $\omega \subset \R^2$ is a bounded domain with boundary of class $\rC^4$.
The remaining part of the fluid boundary is denoted by
$\Gamma_\rf \coloneqq \del \Omega_\rf \setminus \overline{\Gamma_0}$,
and represents the rigid boundary of the fluid domain.
Let $\eta$ denote the normal displacement of the plate.
 Then, at time $t$, the plate occupies the deformed surface
\[
\Gamma_\eta(t) = \{ (s,\eta(t,s)) : s \in \omega \}.
\]
For the thick structural layer, we consider a bounded domain $\Omega_\rs \subset \R^3$ with boundary of class $\rC^2$ such that $\Omega_\rs \cap \Omega_\rf = \emptyset,$ 
$\overline{\Omega}_\rs \cap \overline{\Omega}_\rf = \Gamma_0$.
Its boundary $\del \Omega_\rs$ consists of the flat interface $\Gamma_0$ and the remaining part
$\Gamma_\rs \coloneqq \del \Omega_\rs \setminus \overline{\Gamma_0}$.
The assumed regularity of the boundaries of the reference domains is required in \autoref{sec:analysis of the lin probl} for the analysis of the Stokes operator, the damped plate equation, and the Lam\'e operator.
A schematic illustration of the geometry in the corresponding two-dimensional setting, including both the reference configuration and the moving-domain configuration, is provided in \autoref{fig:geometry}.

\begin{figure}[htbp]
    \centering
    
    \begin{subfigure}[b]{0.45\textwidth}
        \centering
        \begin{tikzpicture}[scale=0.85]
            \fill[solidgreen!15] (-3,0) .. controls (-4,0) and (-3,2.5) .. (0,2.5) 
                                       .. controls (3,2.5) and (4,0) .. (3,0) -- cycle;
            \fill[fluidblue!15] (-3,0) .. controls (-4,0) and (-3,-2.5) .. (0,-2.5) 
                                      .. controls (3,-2.5) and (4,0) .. (3,0) -- cycle;

            \draw[thick, black] (-3,0) -- (3,0);
            \node[below, black] at (-1.5, 0) {$\Gamma_0 = \omega \times \{0\}$};

            \draw[thick, solidgreen] (-3,0) .. controls (-4,0) and (-3,2.5) .. (0,2.5) 
                                           .. controls (3,2.5) and (4,0) .. (3,0);
            \draw[thick, fluidblue] (-3,0) .. controls (-4,0) and (-3,-2.5) .. (0,-2.5) 
                                          .. controls (3,-2.5) and (4,0) .. (3,0);

            \node[black] at (0, 1.3) {$\Omega_\rs$};
            \node[black] at (0, -1.3) {$\Omega_\rf$};
            
            \node[solidgreen, anchor=west] at (2.8, 1.8) {$\Gamma_\rs$};
            \node[fluidblue, anchor=west] at (2.8, -1.8) {$\Gamma_\rf$};

            \fill[black] (-3,0) circle (1.5pt);
            \fill[black] (3,0) circle (1.5pt);
        \end{tikzpicture}
        \caption{Reference configuration.}
    \end{subfigure}
    \hspace{1cm} 
    \begin{subfigure}[b]{0.45\textwidth}
        \centering
        \begin{tikzpicture}[scale=0.85]
            \fill[solidgreen!15] 
                (-3,0) .. controls (-2, 1.2) and (-1, -1.2) .. (0.5, -0.4) .. controls (1.5, 0.4) and (2, 0.8) .. (3,0)
                .. controls (4,0) and (3,2.5) .. (0,2.5) .. controls (-3,2.5) and (-4,0) .. cycle;
                
            \fill[fluidblue!15] 
                (-3,0) .. controls (-2, 1.2) and (-1, -1.2) .. (0.5, -0.4) .. controls (1.5, 0.4) and (2, 0.8) .. (3,0)
                .. controls (4,0) and (3,-2.5) .. (0,-2.5) .. controls (-3,-2.5) and (-4,0) .. cycle;

            \draw[thick, refgray, dashed] (-3,0) -- (3,0);
            \node[below, black] at (-2.2, 0) {$\Gamma_0$};

            \draw[thick, solidgreen] (-3,0) .. controls (-4,0) and (-3,2.5) .. (0,2.5) 
                                           .. controls (3,2.5) and (4,0) .. (3,0);
            \draw[thick, fluidblue] (-3,0) .. controls (-4,0) and (-3,-2.5) .. (0,-2.5) 
                                          .. controls (3,-2.5) and (4,0) .. (3,0);

            \draw[very thick, interfaceorange] 
                (-3,0) .. controls (-2, 1.2) and (-1, -1.2) .. (0.5, -0.4)
                       .. controls (1.5, 0.4) and (2, 0.8) .. (3,0);
            \node[interfaceorange] at (2.2, 0.8) {$\Gamma_\eta(t)$};

            \node[black] at (0, 1.8) {$\Omega_\rs(t)$};
            \node[black] at (0, -1.8) {$\Omega_\rf(t)$};
            
            \node[solidgreen, anchor=west] at (2.8, 1.8) {$\Gamma_\rs$};
            \node[fluidblue, anchor=west] at (2.8, -1.8) {$\Gamma_\rf$};

            \fill[black] (-3,0) circle (1.5pt);
            \fill[black] (3,0) circle (1.5pt);
        \end{tikzpicture}
        \caption{Moving domain problem.}
    \end{subfigure}

    \caption{Sketch of the geometry of the present multilayered FSI problem in the two-dimensional situation.
    The reference configuration with the fixed fluid and thick structure domains $\Omega_\rf$ and $\Omega_\rs$, respectively, and fixed interface $\Gamma_0 = \omega \times \{0\}$ is shown in~(A), while~(B) depicts the moving domain problem with time-dependent domains $\Omega_\rf(t)$ and~$\Omega_\rs(t)$, and with moving interface $\Gamma_\eta(t) = \{(s,\eta(t,s)) : s \in \omega\}$.
    The rigid parts of the fluid and the thick structure boundaries are denoted by $\Gamma_\rf$ and $\Gamma_\rs$, respectively.}
    \label{fig:geometry}
\end{figure}

We consider the three-dimensional viscous, incompressible Navier--Stokes
equations to model the fluid.
Due to the nonlinear fluid--structure interaction, the interface separating the fluid from the thick structural layer is, in general, time-dependent.
Consequently, the fluid domain $\Omega_\rf(t)$ evolves in time and depends on
the displacement of the plate. 
More precisely, for each $t \in (0,T)$, the fluid domain $\Omega_\rf(t)$ is the domain in $\R^3$ whose boundary is given by $\del \Omega_\rf(t) = \Gamma_\rf \cup \Gamma_\eta(t)$, and we may therefore write $\Omega_\rf(t) = \Omega_\rf(\eta(t))$.
For simplicity of notation, we shall usually write $\Omega_\rf(t)$ in the
sequel. Since we are interested in time-periodic solutions, for a fixed period
$T>0$ we study the multilayered fluid--structure interaction problem on the
time interval $(0,T)$.

Let $u^\cF$ and $\pi^\cF$ denote the fluid velocity and pressure on the moving
domain.
Moreover, let $f \colon (0,T) \times \Omega_\rf \to \R^3$ be a given forcing term,
and define $f_\eta \coloneqq f \circ X_\eta$, where $X_\eta \colon
(0,T) \times \Omega_\rf \to (0,T) \times \Omega_\rf(t)$ is a diffeomorphism
mapping the reference domain to the moving fluid domain.
Its precise
construction will be given in \autoref{sec:trafo to fixed dom}.
Then the fluid subproblem reads
\begin{equation}\label{eq:fluid eq}
    \del_t u^\cF + (u^\cF \cdot \nabla) u^\cF
    = \mdiv \sigma_\rf(u^\cF,\pi^\cF) + f_\eta,
    \qquad
    \mdiv u^\cF = 0,
    \quad \tin (0,T) \times \Omega_\rf(t),
\end{equation}
where
\begin{equation}\label{eq:Cauchy stress tensor}
    \sigma_\rf(u^\cF,\pi^\cF) = 2 \mu_\rf \D(u^\cF) - \pi^\cF \Id_3, \twith \D(u^\cF) \coloneqq \frac{1}{2}\bigl(\nabla u^\cF + (\nabla u^\cF)^\top\bigr),
\end{equation}
represents the Cauchy stress tensor for some fluid viscosity parameter $\mu_\rf > 0$.
For simplicity, and as it does not affect the analysis, we will assume that $\mu_\rf \equiv 1$ throughout this paper.

For the thin structure, we assume that its motion is governed by a damped plate equation posed on $\omega$.
More precisely, after rescaling with respect to the thickness and density of the
plate, the displacement $\eta$ satisfies a linear damped plate equation.
Here, $f_\rp$ denotes the force density acting on the structure and will be
specified in the sequel, while
$g \colon (0,T) \times \omega \to \R$ represents an external forcing term.
The plate dynamics are then described by
\begin{equation}\label{eq:plate eq}
    \del_{tt} \eta + \Delta_s^2 \eta - \del_t \Delta_s \eta
    = f_\rp + g,
    \qquad \tin (0,T) \times \omega.
\end{equation}

Moreover, we assume that the deformed plate does not come into contact with the
rigid part of the fluid boundary, that is,
$\Gamma_\rf \cap \Gamma_{\eta}(t) = \emptyset
 \text{ for all } t \in (0,T)$. This non-degeneracy condition will be ensured by the smallness of the plate
displacement.

Finally, for the thick layer, we invoke the equations of linear (visco-)elasticity.
In this case, $d$ denotes the structural displacement of the thick (visco-)elastic layer, and for Lam\'e coefficients $\mu_\rs$, $\lambda_\rs > 0$ and a viscosity parameter $\delta > 0$, the Piola-Kirchhoff stress tensor takes the shape
\begin{equation}\label{eq:Piola-Kirchhoff stress tensor}
    \sigma_\rs(d, \del_t d) = \mu_\rs\bigl((\nabla d + (\nabla d)^\top) + \delta(\nabla \del_t d + (\nabla \del_t d)^\top)\bigr) + \lambda_\rs (\mdiv d + \delta \mdiv \del_t d) \Id_3.
\end{equation}
Note that the viscoelasticity $\delta > 0$ is needed to render the subproblem of the thick structural layer parabolic, see \autoref{ssec:thick structure} for more details.
In the following, in order to ensure strong ellipticity of the associated Lam\'e operator, we assume that
\begin{equation}\label{eq:Lame coeffs}
    \mu_\rs > 0 \tand \mu_\rs + \lambda_\rs > 0.
\end{equation}
The structural problem is already defined in the Lagrangian framework, so it is considered on the fixed domain $\Omega_\rs$, and this domain does not depend on time or the displacement of the plate.
For a suitable forcing term $h \colon (0,T) \times \Omega_\rs \to \R^3$, the equations of linear (visco-)elasticity read as
\begin{equation}\label{eq:lin viscoel eq}
    \del_{tt} d = \mdiv \sigma_\rs(d,\del_t d) + h, \tin (0,T) \times \Omega_\rs.
\end{equation}

The interaction problem is complemented by coupling and boundary conditions.
The coupling conditions consist of kinematic and dynamic ones.
Concerning the former, we assume that there is no slip between the fluid and the thin structure or between the thin and the thick structural layer.
For the unit vector~$\mre_3 = (0,0,1)^\top$, the continuity of velocity reads as
\begin{equation}\label{eq:cont of velocities}
    u^\cF(t,s,\eta(t,s)) = \del_t \eta(t,s) \mre_3, \tfor t \in (0,T), \enspace s \in \omega.
\end{equation}
For $x = (s,0) \in \Gamma_0$, the continuity of displacement between the thick and the thin structure translates as
\begin{equation}\label{eq:cont of displacement}
    d(t,s,0) = \eta(t,s) \mre_3, \tfor t \in (0,T), \enspace s \in \omega.
\end{equation}

The dynamic coupling condition expresses the balance of normal forces acting on
the plate at the fluid--structure interface $\Gamma_{\eta}(t)$.
It enters through the forcing term $f_\rp$ in \eqref{eq:plate eq}, which accounts
for the forces exerted on the plate by the fluid and by the thick structural
layer. Let $n_{\Gamma_{\eta}(t)}$ denote the outer unit normal to the fluid domain along
$\Gamma_{\eta}(t)$, given by
\begin{equation*}
    n_{\Gamma_{\eta}(t)}
    = \frac{1}{\sqrt{1 + |\nabla_s \eta|^2}}
      \binom{-\nabla_s \eta}{1}.
\end{equation*}
With the Cauchy stress tensor $\sigma_\rf$ defined in
\eqref{eq:Cauchy stress tensor} and the Piola--Kirchhoff stress tensor
$\sigma_\rs$ defined in~\eqref{eq:Piola-Kirchhoff stress tensor}, the dynamic
coupling condition can be written as
\begin{equation}\label{eq:dynamic coupling cond}
    f_\rp
    = - \sqrt{1 + |\nabla_s \eta|^2}
      \left.\left(\sigma_\rf(u^\cF,\pi^\cF)
      n_{\Gamma_{\eta}(t)}\right)\right|_{\Gamma_{\eta}(t)} \cdot \mre_3
      + \left.(\sigma_\rs(d,\del_t d)\mre_3)\right|_{\Gamma_0} \cdot \mre_3.
\end{equation}

We now prescribe the boundary conditions for the fluid velocity $u^\cF$ at the rigid part~$\Gamma_\rf$ of the fluid boundary, for the plate at the boundary $\del \omega$ of its domain $\omega$ and for the displacement~$d$ of the thick structure at the non-interface part $\Gamma_\rs$ of the boundary of the thick layer.
Here $u^\cF$ and $d$ satisfy homogeneous Dirichlet boundary conditions, while the plate is clamped, i.e., for the unit exterior normal vector $n_\omega$ to $\del \omega$, we have 
\begin{equation}\label{eq:boundary conditions}
    \begin{aligned}
        u^\cF(t,x) 
        &= 0, &&\tfor t \in (0,T), \enspace x \in \Gamma_\rf,\\
        \eta(t,s) 
        &= \nabla_s \eta(t,s) \cdot n_{\omega} = 0, &&\tfor t \in (0,T), \enspace s \in \del \omega, \tand\\
        d(t,x) 
        &= 0, &&\tfor t \in (0,T), \enspace x \in \Gamma_\rs,
    \end{aligned}
\end{equation}
Finally, the time-periodicity condition reads as
\begin{equation}\label{eq:per data}
    \begin{aligned}
        u^\cF(0,\cdot) 
        &= u^\cF(T,\cdot), \tin \Omega_\rf(\eta(0)), \enspace \eta(0,\cdot) = \eta(T,\cdot) \tand \del_t \eta(0,\cdot) = \del_t \eta(T,\cdot), \tin \omega,\\
        d(0,\cdot) 
        &= d(T,\cdot) \tand \del_t d(0,\cdot) = \del_t d(T,\cdot), \tin \Omega_\rs.
    \end{aligned}
\end{equation}

At this stage, let us also remark that we take into account the volume constraint 
\begin{equation}\label{eq:volume constraint displacement}
    \int_{\omega} \eta \srd s = 0.
\end{equation}
Thus, we recall the space $\rL_\rm^q(\omega)$ of functions with average zero in $\rL^q(\omega)$. 
In addition, we invoke the orthogonal projection $P_\rm \colon \rL^q(\omega) \to \rL_\rm^q(\omega)$ onto the space of functions with average zero.
It is given by
\begin{equation}\label{eq:proj P_m}
    P_\rm f \coloneqq f - \frac{1}{|\omega|} \int_\omega f \srd s, \tfor f \in \rL^q(\omega).
\end{equation}
We then apply this projection to the plate equation \eqref{eq:plate eq}, leading to the equations
\begin{equation}\label{eq:projected plate eq}
    \begin{aligned}
        \del_{tt} \eta + P_\rm \Delta_s^2 \eta - \Delta_s \del_t \eta 
        &= -P_\rm \left(\sqrt{1 + |\nabla_s \eta|^2} \left.\left(\sigma_\rf(u^\cF,\pi^\cF) n_{\Gamma_{\eta}(t)}\right) \right|_{\Gamma_{\eta}(t)} \cdot \mre_3\right)\\
        &\quad + P_\rm\left(\left.(\sigma_\rs(d,\del_t d) \mre_3)\right|_{\Gamma_0} \cdot \mre_3\right) + P_\rm g
    \end{aligned}
\end{equation}
and
\begin{equation}\label{eq:average part plate eq}
    \begin{aligned}
        \int_\omega \pi^\cF(t,s,\eta(t,s)) \srd s
        &= \int_\omega \Delta_s^2 \eta(t,s) \srd s + \int_\omega \sqrt{1 + |\nabla_s \eta|^2} \left.\left(2 \D(u^\cF) n_{\Gamma_{\eta}(t)}\right) \right|_{\Gamma_{\eta}(t)} \cdot \mre_3 \srd s\\
        &\quad - \int_\omega \left.(\sigma_\rs(d,\del_t d) \mre_3)\right|_{\Gamma_0} \cdot \mre_3 \srd s - \int_\omega g \srd s.
    \end{aligned}
\end{equation}
This allows us to determine the pressure precisely, not only up to a constant.
As a consequence, we only keep \eqref{eq:projected plate eq} together with the volume constraint \eqref{eq:volume constraint displacement}, so we solve the pressure up to a constant.
Finally, we employ \eqref{eq:average part plate eq} to determine the constant in the pressure.

We now summarize the multilayered FSI problem consisting of the respective fluid, plate and linearly (visco-)elastic subproblem \eqref{eq:fluid eq}, \eqref{eq:plate eq} and \eqref{eq:lin viscoel eq} as well as the coupling conditions \eqref{eq:cont of velocities}--\eqref{eq:dynamic coupling cond}, the boundary conditions \eqref{eq:boundary conditions} and the time-periodicity conditions \eqref{eq:per data}.
The resulting problem, with \eqref{eq:projected plate eq} and~\eqref{eq:volume constraint displacement} replacing \eqref{eq:plate eq}, and with $\cJ_\eta \coloneqq \sqrt{1 + |\nabla_s \eta|^2}$, reads as
\begin{equation}\label{eq:multilayered fsi}
    \left\{
    \begin{aligned}
        \del_t u^\cF + (u^\cF \cdot \nabla) u^\cF
        &= \mdiv \sigma_\rf(u^\cF,\pi^\cF) + f_\eta, \enspace \mdiv u^\cF = 0, &&\tin (0,T) \times \Omega_\rf(t),\\
        \del_{tt} \eta + P_\rm \Delta_s^2 \eta - \del_t \Delta_s \eta
        &= -P_\rm\left(\cJ_\eta \left.\left(\sigma_\rf(u^\cF,\pi^\cF) n_{\Gamma_{\eta}(t)}\right) \right|_{\Gamma_{\eta}(t)} \cdot \mre_3\right)\\
        &\qquad + P_\rm \left(\left.(\sigma_\rs(d,\del_t d) \mre_3)\right|_{\Gamma_0} \cdot \mre_3\right) + P_\rm g, &&\tin (0,T) \times \omega,\\
        \int_\omega \eta(t,s) \srd s 
        &= 0, &&\tfor t \in (0,T),\\
        \del_{tt} d 
        &= \mdiv \sigma_\rs(d,\del_t d) + h, &&\tin (0,T) \times \Omega_\rs,\\
        u^\cF(t,s,\eta(t,s)) 
        &= \del_t \eta(t,s) \mre_3, &&\tfor t \in (0,T), \enspace s \in \omega,\\
        d(t,s,0) 
        &= \eta(t,s) \mre_3, &&\tfor t \in (0,T), \enspace s \in \omega,\\
        u^\cF(t,x) 
        &= 0, &&\tfor t \in (0,T), \enspace x \in \Gamma_\rf,\\
        d(t,x) 
        &= 0, &&\tfor t \in (0,T), \enspace x \in \Gamma_\rs,\\
        \eta = \nabla_s \eta \cdot n_{\omega}
        &= 0, &&\ton (0,T) \times \del \omega,\\
        u^\cF(0,\cdot) 
        &= u^\cF(T,\cdot), &&\tin \Omega_\rf(\eta(0)),\\
        \eta(0,\cdot) 
        &= \eta(T,\cdot) \tand \del_t \eta(0,\cdot) = \del_t \eta(T,\cdot), &&\tin \omega,\\
        d(0,\cdot) 
        &= d(T,\cdot) \tand \del_t d(0,\cdot) = \del_t d(T,\cdot), &&\tin \Omega_\rs.
    \end{aligned}
    \right.
\end{equation}

Before stating the corresponding theorem, we require some preparation in terms of notation.
More precisely, let $\eta \in \rL^p(0,T;\rW^{4,q}(\omega)) \cap \rW^{2,p}(0,T;\rL^q(\omega))$ such that $\eta = 0$ and $\nabla_s \eta \cdot n_\omega = 0$ is valid on $\del \omega$, and we have $\Gamma_\rf \cap \Gamma_{\eta(t)} = \emptyset$ for all $t \in [0,T]$.
In \autoref{sec:trafo to fixed dom}, the existence of a map $X = X_\eta$ being a $\rC^1$-diffeomorphism from~$\Omega_\rf$ onto $\Omega_\rf(t)$ and satisfying $X \in \rL^p\bigl(0,T;\rW^{2,q}(\Omega_\rf)^3\bigr) \cap \rW^{2,p}\bigl(0,T;\rL^q(\Omega_\rf)^3\bigr)$ is shown.
Let us recall the definition of function spaces on moving domains from \autoref{sec:not and function spaces}.
In order to formulate the main result in a more concise way, we introduce the following space.
For $v^\cF = (u^{\cF},\pi^{\cF},\eta,\del_t \eta,d,\del_t d)$, we define
\begin{equation*}
    \begin{aligned}
        \E_1^* 
        \coloneqq \bigl\{v^\cF : \, &u^{\cF} \in \rW^{1,p}(0,T;\rL^q(\Omega_\rf(\cdot))^3) \cap \rL^p(0,T;\rW^{2,q}(\Omega_\rf(\cdot))^3), \, \pi^{\cF} \in \rL^p(0,T;\rW^{1,q}(\Omega_\rf(\cdot)) \cap \rL_\rm^q(\Omega_\rf(\cdot))),\\
        &\eta \in \rW^{1,p}(0,T;\rW^{2,q}(\omega)) \cap \rL^p(0,T;\rW^{4,q}(\omega)), \enspace \del_t \eta \in \rW^{1,p}(0,T;\rL^q(\omega)) \cap \rL^p(0,T;\rW^{2,q}(\omega)),\\
        &d \in \rW^{1,p}(0,T;\rW^{1,q}(\Omega_\rs)^3), \enspace \del_t d \in \rW^{1,p}(0,T;\rL^q(\Omega_\rs)^3) \cap \rL^p(0,T;\rW^{1,q}(\Omega_\rs)^3) \tand\\ 
        &d + \delta \del_t d \in \rL^p(0,T;\rW^{2,q}(\Omega_\rs)^3)\bigr\}.
    \end{aligned}
\end{equation*}
Before stating the main result, we make precise the notion of a strong $T$-periodic solution to the nonlinearly coupled multilayered fluid-structure interaction problem \eqref{eq:multilayered fsi}.

\begin{defn}\label{def:strong sol}
We say that $v^\cF = (u^{\cF},\pi^{\cF},\eta,\del_t \eta,d,\del_t d)$ is a strong $T$-periodic solution to \eqref{eq:multilayered fsi} if 
\begin{enumerate}[(a)]
    \item it satisfies $v^\cF \in \E_1^*$, 
    \item the fluid equation \eqref{eq:multilayered fsi}$_1$, thin structure equation \eqref{eq:multilayered fsi}$_{2-3}$ and thick structure equation \eqref{eq:multilayered fsi}$_5$ are satisfied in a strong $\rL^p$-sense, 
    \item the volume constraint \eqref{eq:multilayered fsi}$_4$ is fulfilled,
    \item the kinematic coupling conditions \eqref{eq:multilayered fsi}$_{6-7}$ and boundary conditions \eqref{eq:multilayered fsi}$_{8-10}$ hold in a trace sense, and
    \item the periodicity conditions in \eqref{eq:multilayered fsi}$_{11-13}$ are satisfied.
\end{enumerate}
\end{defn}

The main result of this paper below asserts the existence of a unique time-periodic strong solution to the multilayered fluid-structure interaction problem as made precise in \eqref{eq:multilayered fsi} provided the time-periodic forcing terms are sufficiently small.
The condition on $p$ and $q$ is required to establish suitable nonlinear estimates by means of Sobolev embeddings, see \autoref{ssec:nonlin ests}.

\begin{thm}\label{thm:ex of strong per sol}
Let $p$, $q \in (1,\infty)$ be such that $\frac{1}{p} + \frac{3}{2q} < \frac{3}{2}$, and assume that $f \in \rL^p(0,T;\rL^q(\Omega_\rf)^3)$, $g \in \rL^p(0,T;\rL^q(\omega))$ and $h \in \rL^p(0,T;\rL^q(\Omega_\rs)^3)$.

Then there is $R_1 > 0$ so that for every $R \in (0,R_1)$, there exists $\eps = \eps(R) > 0$ such that if
\begin{equation*}
    \| f \|_{\rL^p(0,T;\rL^q(\Omega_\rf))} + \| g \|_{\rL^p(0,T;\rL^q(\omega))} + \| h \|_{\rL^p(0,T;\rL^q(\Omega_\rs))} < \eps,
\end{equation*}
then the multilayered FSI problem \eqref{eq:multilayered fsi} has a unique strong solution $(u^{\cF},\pi^{\cF},\eta,\del_t \eta,d,\del_t d) \in \oB_{\E_1^*}(0,R)$ in the sense of \autoref{def:strong sol}, where $\oB_{\E_1^*}(0,R)$ is the closed ball in $\E_1^*$ with center zero and radius $R$.
Besides, $\Gamma_\rf \cap \Gamma_\eta(t) = \emptyset$ for all~$t \in [0,T]$, i.e., the geometry does not degenerate, and $(u^{\cF},\pi^{\cF},\eta,\del_t \eta,d,\del_t d)$ satisfies
\begin{equation*}
    \begin{aligned}
        u^{\cF} 
        &\in \rC([0,T];\rB_{qp}^{2-\nicefrac{2}{p}}(\Omega_\rf(\cdot))^3), \enspace \eta \in \rC([0,T];\rB_{qp}^{4-\nicefrac{2}{p}}(\omega)), \enspace \del_t \eta \in \rC([0,T];\rB_{qp}^{2-\nicefrac{2}{p}}(\omega)),\\
        d 
        &\in \rC([0,T];\rW^{1,q}(\Omega_\rs)^3), \enspace \del_t d \in \rC([0,T];\rB_{qp}^{1-\nicefrac{1}{p}}(\Omega_\rs)^3) \tand d + \delta \del_t d \in \rC([0,T];\rB_{qp}^{2-\nicefrac{2}{p}}(\Omega_\rs)^3).
    \end{aligned}
\end{equation*}
\end{thm}

At this stage, a remark on the role of the viscosity in the thick structural layer is in order.

\begin{rem}
As indicated above, the presence of viscoelastic damping in the structural model, i.e., assuming~$\delta > 0$ in \eqref{eq:Piola-Kirchhoff stress tensor}, is used to render this subproblem fully parabolic.
However, it seems that this assumption is also essential for establishing the existence of {\em time-periodic weak solutions}, see \cite[Rem.~1.17]{MR25}, where it is needed to obtain a diffusion estimate.
For more background on the occurrence of resonance and the role of damping in time-periodic problems involving elastic structures, we also refer to \cite{GMZZ:14}.
\end{rem}

\section{Transformation to a fixed domain}\label{sec:trafo to fixed dom}

In this section, we describe the procedure to transform the multilayered, nonlinearly coupled fluid-structure interaction problem \eqref{eq:multilayered fsi} with time-dependent fluid-domain $\Omega_\rf(t)$ to a problem on a fixed domain.

Without loss of generality, and following \cite{MT:21}, we may assume that there is $\alpha >0$ with $\omega \times (-\alpha,0) \subset \Omega_\rf$ and $\omega \times (0,\alpha) \subset \R^3 \setminus \Omega_\rf$.
Moreover, we fix a cut-off function $\psi\in \rC_{\mathrm{c}}^{\infty}\left(\mathbb{R}^{3};\left[0,1\right]\right)$ satisfying
\begin{equation*}
    \psi(z) =\left\{
    \begin{aligned}
        1, &\tfor z \in \left(-\frac{\alpha}{2},\frac{\alpha}{2}\right),\\
        0, &\tfor z \in \R \setminus (-\alpha,\alpha).
    \end{aligned}
    \right.
\end{equation*}
By extending $\eta$ from $\omega$ to $\R^{2}$ by zero, which leads to $\eta \in \rC_\mathrm{c}^1(\R^2)$, we may define a diffeomorphism 
\begin{equation}\label{eq:diffeo X}
    X(y_1,y_2,y_3) = X_\eta(y_1,y_2,y_3) \coloneqq \begin{pmatrix}
        y_1\\ y_2\\ y_3 + \psi(y_3) \eta(y_1,y_2)
    \end{pmatrix}. 
\end{equation}
In the following, it turns out to be advantageous to make a smallness assumption on the plate displacement~$\eta$.
In fact, for $\eta \in \rC^1(\overline{\omega})$ satisfying the clamped boundary conditions $\eta = 0$ and $\nabla_s \eta \cdot n_\omega = 0$, we suppose that there exists $\delta_0 > 0$ such that
\begin{equation}\label{eq:smallness ass displacement}
    \| \eta \|_{\rL^\infty(\omega)} \le \delta_0.
\end{equation}
We choose $\delta_0$ in \eqref{eq:smallness ass displacement} as
\begin{equation}\label{eq:concr bound for plate displacement}
    \delta_0 \coloneqq \frac{1}{2 \| \psi' \|_{\rL^\infty(\R)}}.
\end{equation}
Let $\eta = \eta(t)$ be a time-dependent plate displacement satisfying the
clamped boundary conditions
$\eta = 0$ and 
$\nabla_s \eta \cdot n_\omega = 0$ on $\del \omega$,
and assume that $\eta$ fulfills the smallness condition
\eqref{eq:smallness ass displacement} with this choice of $\delta_0$ for all
$t \in [0,T]$. If $\eta(t) \in \rC^1(\overline{\omega})$ for each $t \in [0,T]$,
we set $X(t,\cdot) \coloneqq X_{\eta(t)}$ and $X(t,\cdot)$ is a $\rC^1$-diffeomorphism from $\Omega_\rf$ onto
$\Omega_\rf(t)$ for all $t \in [0,T]$. Moreover, 
we observe that $\eta \in \rC([0,T];\rB_{qp}^{4-\nicefrac{2}{p}}(\omega))$. Under the assumptions of \autoref{thm:ex of strong per sol} on $p$ and $q$, it follows from Sobolev embeddings that $\eta \in \rC^1(\overline{\omega})$ for all $t \in [0,T]$. We refer to the proof of \autoref{lem:ests of plate and fluid vars} for details.

In addition, for~$t \in [0,T]$, we use $Y(t,\cdot) = X(t,\cdot)^{-1}$ to denote the inverse of $X(t,\cdot)$.
Note that $X \in \rC_\mathrm{b}^0([0,T];\rC^1(\overline{\Omega_\rf})^3)$.
Moreover, by elementary computations, for $t \in (0,T)$ and $y = (y_1,y_2,y_3) \in \omega \times (-\nicefrac{\alpha}{2},\nicefrac{\alpha}{2})$, we get
\begin{equation*}
    \det \nabla X(t,y) = 1 \tand \Cof(\nabla X)(t,y) = \begin{pmatrix}
        1 & 0 & -\del_{y_1} \eta(t,y_1,y_2)\\
        0 & 1 & -\del_{y_2} \eta(t,y_1,y_2)\\
        0 & 0 & 1
    \end{pmatrix}.
\end{equation*}
We may thus set 
\begin{equation*}
    u(t,y) \coloneqq \Cof \nabla X^\top(t,y) u^{\cF}(t,X(t,y)) \tand \pi(t,y) \coloneqq \Tilde{\pi}(t,X(t,y)), \tfor (t,y) \in (0,T) \times \Omega_\rf.
\end{equation*}
With this, also invoking the relation of $f_\eta = f \circ X_\eta$ and $f$, we get the system in reference configuration
\begin{equation}\label{eq:transformed syst}
\left\{
    \begin{aligned}
        \del_t u 
        &= \mdiv \sigma_\rf(u,\pi)+F(u,\pi,\eta) + f, \enspace \mdiv u = 0, &&\tin (0,T) \times \Omega_\rf,\\
        \del_{tt} \eta + P_\rm \Delta_s^2 \eta - \del_t \Delta_s \eta
        &= -P_\rm\left(\left.(\sigma_\rf(u,\pi) \mre_3) \right|_{\Gamma_0} \cdot \mre_3\right)\\
        &\qquad +P_\rm \left(G\left(u, \pi, \eta \right) \right)\\
        &\qquad + P_\rm \left(\left.(\sigma_\rs(d,\del_t d) \mre_3)\right|_{\Gamma_0} \cdot \mre_3\right) + P_\rm g, &&\tin (0,T) \times \omega,\\
        \int_\omega \eta(t,s) \srd s 
        &= 0, &&\tfor t \in (0,T),\\
        \del_{tt} d 
        &= \mdiv \sigma_\rs(d,\del_t d) + h, &&\tin (0,T) \times \Omega_\rs,\\
        u(t,s,0) 
        &= \del_t \eta(t,s) \mre_3, &&\tfor t \in (0,T), \enspace s \in \omega,\\
        d(t,s,0) 
        &= \eta(t,s) \mre_3, &&\tfor t \in (0,T), \enspace s \in \omega,\\
        u(t,x) 
        &= 0, &&\tfor t \in (0,T), \enspace x \in \Gamma_\rf,\\
        d(t,x) 
        &= 0, &&\tfor t \in (0,T), \enspace x \in \Gamma_\rs,\\
        \eta = \nabla_s \eta \cdot n_{\omega}
        &= 0, &&\ton (0,T) \times \del \omega,\\
        u(0,\cdot) 
        &= u(T,\cdot), &&\tin \Omega_\rf,\\
        \eta(0,\cdot) 
        &= \eta(T,\cdot) \tand \del_t \eta(0,\cdot) = \del_t \eta(T,\cdot), &&\tin \omega,\\
        d(0,\cdot) 
        &= d(T,\cdot) \tand \del_t d(0,\cdot) = \del_t d(T,\cdot), &&\tin \Omega_\rs.
    \end{aligned}
    \right.
\end{equation} 

In the sequel, we also write $a \coloneqq \Cof(\nabla Y)^\top$ and $b \coloneqq \Cof(\nabla X)^\top$, yielding that
\begin{equation*}
    u(t,y) = b(t,y) u^{\cF}(t,X(t,y)) \tand u^{\cF}(t,x) = a(t,x) u(t,Y(t,x)).
\end{equation*}
The transformed terms $F=\left(F_{\alpha}\right)_{1\le\alpha\le3}$ and $G$ in \eqref{eq:transformed syst} are then given by
\begin{equation}\label{eq:RHS F}
    \begin{aligned}
        F_{\alpha}\left(u,\pi,\eta\right)
        &= \sum_{i,j,k=1}^3 b_{\alpha i}\frac{\partial^{2}a_{ik}}{\partial x_{j}^{2}}\left(X\right)u_{k}+2\sum_{i,j,k,l=1}^3 b_{\alpha i}\frac{\partial a_{ik}}{\partial x_{j}}\left(X\right)\frac{\partial u_{k}}{\partial y_{l}}\frac{\partial Y_{l}}{\partial x_{j}}\left(X\right)\\
        &\quad  +\sum_{j,l,m=1}^3\frac{\partial^{2}u_{\alpha}}{\partial y_{l}\partial y_{m}}\left(\frac{\partial Y_{l}}{\partial x_{j}}\left(X\right)\frac{\partial Y_{m}}{\partial x_{j}}\left(X\right)-\delta_{l,j}\delta_{m,j}\right)+\sum_{j,l=1}^3 \frac{\partial u_{\alpha}}{\partial y_{l}}\frac{\partial^{2}Y_{l}}{\partial x_{j}^{2}}\left(X\right)\\
        &\quad -\sum_{i,k=1}^3 \frac{\partial\pi}{\partial y_{k}}\left(\det\left(\nabla X\right)\left(\frac{\partial Y_{\alpha}}{\partial x_{i}}\left(X\right)\frac{\partial Y_{k}}{\partial x_{i}}\left(X\right)-\delta_{\alpha,i}\delta_{k,i}\right)\right)\\
        &\quad -\sum_{i,j,k,m=1}^3 b_{\alpha i}\frac{\partial a_{ik}}{\partial x_{j}}\left(X\right)a_{jm}\left(X\right)u_{k}u_{m}-\frac{1}{\det\left(\nabla X\right)}\left[\left(u\cdot\nabla\right)u\right]_{\alpha}\\
        &\quad -\left[b\left(\partial_{t}a\right)\left(X\right)u\right]_{\alpha}-\left[\left(\nabla u\right)\left(\partial_{t}Y\right)\left(X\right)\right]_{\alpha}
    \end{aligned}
\end{equation}
and
\begin{equation}\label{eq:RHS G}
    \begin{aligned}
        G\left(u,\pi,\eta\right)
        &= \sum_{k=1}^{3}\left(\sum_{i=1}^{2}\partial_{s_{i}}\eta\left(\frac{\partial a_{ik}}{\partial x_{3}}\left(X\right)+\frac{\partial a_{3k}}{\partial x_{i}}\left(X\right)\right)-2\frac{\partial a_{3k}}{\partial x_{3}}\left(X\right)\right)u_{k}\\
        &\quad +\sum_{k=1}^{3}\biggl(\sum_{i=1}^{2}\partial_{s_{i}}\eta\left(a_{ik}\left(X\right)\frac{\partial Y_{l}}{\partial x_{3}}\left(X\right)+a_{3k}\left(X\right)\frac{\partial Y_{l}}{\partial x_{i}}\left(X\right)\right)\\
        &\quad -2\left(a_{3k}\left(X\right)\frac{\partial Y_{l}}{\partial x_{3}}\left(X\right)-\delta_{3,k}\delta_{3,l}\right)\frac{\partial u_{k}}{\partial y_{l}}\biggr).
\end{aligned}
\end{equation}

After providing the transformed system of equations in the reference configuration, we reformulate the existence result of a strong time-periodic solution in this configuration.
For this purpose, let us introduce some pieces of notation for the function spaces.
In fact, we set
\begin{equation*}
    \begin{aligned}
        \E_1^u 
        &\coloneqq \rW^{1,p}(0,T;\rL^q(\Omega_\rf)^3) \cap \rL^p(0,T;\rW^{2,q}(\Omega_\rf)^3), \enspace \E_1^\pi \coloneqq \rL^p(0,T;\rW_\rm^{1,q}(\Omega_\rf)),\\
        \E_1^{\eta_1}
        &\coloneqq \rW^{1,p}(0,T;\rW^{2,q}(\omega)) \cap \rL^p(0,T;\rW^{4,q}(\omega)) \cap \rL^p(0,T;\rL_\rm^q(\omega)),\\
        \E_1^{\eta_2} 
        &\coloneqq \rW^{1,p}(0,T;\rL^q(\omega)) \cap \rL^p(0,T;\rW^{2,q}(\omega)) \cap \rL^p(0,T;\rL_\rm^q(\omega)),\\
        \E_1^{d_1}
        &\coloneqq \rW^{1,p}(0,T;\rW^{1,q}(\Omega_\rs)^3) \tand \E_1^{d_2} \coloneqq \rW^{1,p}(0,T;\rL^q(\Omega_\rs)^3) \cap \rL^p(0,T;\rW^{1,q}(\Omega_\rs)^3).
    \end{aligned}
\end{equation*}
In the sequel, we will use $v$ to denote the principle variable, i.e., $v = (u,\pi,\eta_1,\eta_2,d_1,d_2)$.
We then define
\begin{equation}\label{eq:max reg space}
    \begin{aligned}
        \tE_1 
        &\coloneqq \bigl\{(u,\eta_1,\eta_2,d_1,d_2) \in \E_1^u \times \E_1^{\eta_1} \times \E_1^{\eta_2} \times \E_1^{d_1} \times \E_1^{d_2} : d_1 + \delta d_2 \in \rL^p(0,T;\rW^{2,q}(\Omega_\rs)^3)\bigr\} \tand\\
        \E_1 
        &\coloneqq \bigl\{v=(u,\pi,\eta_1,\eta_2,d_1,d_2) : (u,\eta_1,\eta_2,d_1,d_2) \in \tE_1 \tand \pi \in \E_1^\pi\bigr\}.
    \end{aligned}
\end{equation}
The main result stated in terms of the reference configuration is now given as follows.
Let us note that the difference with the spaces in \autoref{thm:ex of strong per sol} is that the the latter spaces concern the moving domain problem, i.e., they are defined by invoking the diffeomorphism $X$ from $\Omega_\rf$ to $\Omega_\rf(t)$.
Moreover, note that the definition of a strong solution to \eqref{eq:transformed syst} is analogous to \autoref{def:strong sol} in the context of the problem on the moving domain.

\begin{thm}\label{thm:strong time per sol in ref config}
Let $p$, $q \in (1,\infty)$ satisfy $\frac{1}{p} + \frac{3}{2q} < \frac{3}{2}$, and $f \in \rL^p(0,T;\rL^q(\Omega_\rf)^3)$, $g \in \rL^p(0,T;\rL^q(\omega))$ as well as $h \in \rL^p(0,T;\rL^q(\Omega_\rs)^3)$.

Then there exists $R_1 > 0$ such that for all $R \in (0,R_1)$, there is $\eps = \eps(R) > 0$ so that if
\begin{equation*}
    \| f \|_{\rL^p(0,T;\rL^q(\Omega_\rf))} + \| g \|_{\rL^p(0,T;\rL^q(\omega))} + \| h \|_{\rL^p(0,T;\rL^q(\Omega_\rs))} < \eps,
\end{equation*}
there exists a unique strong solution $ (u,\pi,\eta,\del_t \eta,d,\del_t d) \in \oB_{\E_1}(0,R)$, where $\E_1$ has been made precise in~\eqref{eq:max reg space}, to \eqref{eq:transformed syst}.
Besides, the displacement $\eta$ fulfills \eqref{eq:smallness ass displacement}, and $X(t,\cdot) \colon \Omega_\rf \to \Omega_\rf(t)$ is a $\rC^1$-diffeomorphism for all $t \in [0,T]$.
\end{thm}

\section{Analysis of the linearized problem}\label{sec:analysis of the lin probl}

The present section is of central importance.
In fact, its aim is to establish the so-called {\em maximal periodic $\rL^p$-regularity} via the Arendt-Bu theorem \cite{AB:02}.
For a generator $A \colon \rX_1 \subset \rX_0 \to \rX_0$ of a $\rC_0$-semigroup on $\rX_0$, the latter result provides a characterization of the well-posedness of the inhomogeneous periodic abstract Cauchy problem
\begin{equation}\label{eq:gen per ACP}
    \left\{
    \begin{aligned}
        w'(t) - A w(t)
        &= f(t), \enspace \tfor  t \in (0,2\pi),\\
        w(0)
        &= w(2 \pi),
    \end{aligned}
    \right.
\end{equation}
in terms of the well-posedness of the associated {\em initial value problem} and a spectral condition on the $\rC_0$-semigroup.
More precisely, setting $\F \coloneqq \rL^p(0,2 \pi;\rX_0)$ and $\E \coloneqq \rW^{1,p}(0,2 \pi;\rX_0) \cap \rL^p(0,2\pi;\rX_1)$, we recall that maximal periodic $\rL^p$-regularity means that for every $f \in \F$, there exists a unique solution~$w \in \E$ to~\eqref{eq:gen per ACP}.
Furthermore, the closed graph theorem then implies the existence of a constant $C > 0$ such that
\begin{equation*}
    \| w \|_{\E} \le C \cdot \| f \|_{\F}.
\end{equation*}
The Arendt-Bu theorem is given as follows.

\begin{prop}[{\cite[Thm.~5.1]{AB:02}}]\label{prop:Arendt-Bu thm}
Consider a Banach space $\rX_0$, and let $A \colon \rX_1 \to \rX_0$ be the generator of a $\rC_0$-semigroup on $\rX_0$.
Then $A$ has maximal periodic $\rL^p$-regularity if and only if $A$ has maximal $\rL^p$-regularity, and the semigroup satisfies $1 \in \rho(\mre^{2 \pi A})$.
Maximal $\rL^p$-regularity means that for every $f \in \F$, there is a unique solution $w \in \E$ to the abstract Cauchy problem
\begin{equation*}
    \left\{
    \begin{aligned}
        w'(t) - A w(t)
        &= f(t), \enspace \tfor t \in (0,2\pi),\\
        w(0)
        &= 0.
    \end{aligned}
    \right.
\end{equation*}
\end{prop}

A simple rescaling argument reveals that the Arendt-Bu theorem is also valid for general time intervals~$(0,T)$.
On the other hand, the spectral mapping theorem for generators of analytic semigroups, see, e.g., \cite[Cor.~IV.3.12]{EN:00}, and the fact that the maximal $\rL^p$-regularity already implies the generation of an analytic semigroup yield that the spectral condition $1 \in \rho(\mre^{2 \pi A})$ is equivalent with $0 \in \rho(A)$.

This section is dedicated to showing that the linearized problem that corresponds to the transformed problem \eqref{eq:transformed syst} lies within the scope of the Arendt-Bu theorem as recalled in \autoref{prop:Arendt-Bu thm}.
In fact, setting $\eta_1 \coloneqq \eta$, $\eta_2 \coloneqq \del_t \eta$, $d_1 \coloneqq d$ and $d_2 \coloneqq \del_t d$ in order to rewrite the linearized problem as a first order problem, we find that the linearized problem under consideration is given by
\begin{equation}\label{eq:lin per multilayered interaction probl}
    \left\{
    \begin{aligned}
        \del_t u - \mdiv \sigma_\rf(u,\pi)
        &= f, \enspace \mdiv u = 0, &&\tin (0,T) \times \Omega_\rf,\\
        \del_t \eta_1
        &= \eta_2, &&\tin (0,T) \times \omega,\\
        \del_t \eta_2 + P_\rm \Delta_s^2 \eta_1 - \Delta_s \eta_2
        &= -P_\rm \left(\left.\sigma_\rf(u,\pi) \right|_{\Gamma_0} \mre_3 \cdot \mre_3\right)\\
        &\qquad + P_\rm \left(\left.(\sigma_\rs(d_1,d_2) \mre_3)\right|_{\Gamma_0} \cdot \mre_3\right) + P_\rm g, &&\tin (0,T) \times \omega,\\
        \int_\omega \eta_1(t,s) \srd s 
        &= 0, &&\tfor t \in (0,T),\\
        \del_t d_1
        &= d_2, &&\tin (0,T) \times \Omega_\rs,\\
        \del_t d_2 - \mdiv \sigma_\rs(d_1,d_2)
        &= h, &&\tin (0,T) \times \Omega_\rs,\\
        u
        &= P_\rm(\eta_2) \mre_3, &&\ton (0,T) \times \Gamma_0,\\
        d_1
        &= P_\rm(\eta_1) \mre_3, &&\ton (0,T) \times \Gamma_0,\\
        u
        &= 0, &&\ton (0,T) \times \Gamma_\rf,\\
        d_1
        &= 0, &&\ton (0,T) \times \Gamma_\rs,\\
        \eta_1 = \nabla_s \eta_1 \cdot n_{\omega}
        &= 0, &&\ton (0,T) \times \del \omega,\\
        u(0,\cdot) 
        &= u(T,\cdot), &&\tin \Omega_\rf,\\
        \eta_1(0,\cdot) 
        &= \eta_1(T,\cdot) \tand \eta_2(0,\cdot) = \eta_2(T,\cdot), &&\tin \omega,\\
        d_1(0,\cdot) 
        &= d_1(T,\cdot) \tand d_2(0,\cdot) = d_2(T,\cdot), &&\tin \Omega_\rs.
    \end{aligned}
    \right.
\end{equation}

We remark that the term $P_\rm \left(\left.\sigma_\rf(u,\pi) \right|_{\Gamma_0} \mre_3 \cdot \mre_3\right)$ simplifies significantly due to the the divergence-free condition $\mdiv u = 0$ in $\Omega_\rf$ as well as the fact that $u$ vanishes in horizontal direction on $\Gamma_0$ with regard to the condition $u = P_\rm(\eta_2) \mre_3$.
This leads to $\left.(\nabla u + (\nabla u)^\top)\right|_{\Gamma_0} = 0$.
As a result, with the modified trace operator
\begin{equation}\label{eq:mod trace op}
    \gamma_\rm f \coloneqq P_\rm(\left.f \right|_{\Gamma_0}) = f(\cdot,0) - \frac{1}{|\omega|} \int_{\omega} f(s',0) \srd s', \tfor f \in \rW^{r,q}(\Omega_\rf), \enspace r > \nicefrac{1}{q},
\end{equation}
we obtain 
\begin{equation}\label{eq:more comp repr kinematic cc in plate eq}
    -P_\rm \left(\left.\sigma_\rf(u,\pi) \right|_{\Gamma_0} \mre_3 \cdot \mre_3\right) = \gamma_\rm \pi.
\end{equation}

In the following subsections, we shall be concerned with providing a result on the solvability of \eqref{eq:lin per multilayered interaction probl}.
Note that even though \eqref{eq:lin per multilayered interaction probl} is a {\em time-periodic problem}, we will mainly work with the {\em initial value problem}, since we will use the aforementioned Arendt-Bu theorem \autoref{prop:Arendt-Bu thm}.

\subsection{Preliminary considerations}\label{ssec:prelim cons}
\ 

This subsection is of preparatory character.
First, we recall some concepts that will be important in the sequel such as $\cR$-sectoriality or the boundedness of the $\Hinfty$-calculus.
Moreover, we recall properties of the Stokes operator, and we take into account the block operator matrix associated with the damped plate block.

We start by making precise what it means for a Banach space $\rX$ to be a so-called {\em UMD space}.
In fact, UMD represents {\em unconditional martingale differences}, and while it is a probabilistic concept, it can be characterized in terms of the Hilbert transform, namely, $\rX$ is a UMD space if and only if the Hilbert transform is bounded on $\rL^p(\R;\rX)$ for some $p \in (1,\infty)$.
Many classical function spaces such as $\rL^p$-spaces, Sobolev spaces $\rW^{s,p}$, Bessel potential spaces $\rH^{s,p}$ or Besov spaces $\rB_{pq}^s$ have this property provided $p$, $q \in (1,\infty)$, and it carries over to closed subspaces of these spaces.
We also refer to \cite[Sec.~III.4.4 and III.4.5]{Ama:95} for more background on this concept.

Next, we invoke the concept of $\cR$-sectoriality for an operator.
For this purpose, we first require the notion of $\cR$-boundedness.
For Banach spaces $\rX$, $\rY$, a family of operators $\cT(\rX,\rY)$ from $\rX$ to $\rY$ is said to be {\em $\cR$-bounded} if there exist $C > 0$ and $p \in [1,\infty)$ such that for all $N \in \N$, $T_j \in \cT$, $x_j \in \rX$, and for every independent, symmetric, $\{-1,1\}$-valued random variable $\eps_j$ on a probability space $(\Omega,\mathcal{A},\mu)$, it holds that
\begin{equation}\label{eq:ineq R-bddness}
    \left\| \sum_{j=1}^N \eps_j T_j x_j \right\|_{\rL^p(\Omega;\rY)} \le C \cdot \left\| \sum_{j=1}^N \eps_j x_j \right\|_{\rL^p(\Omega;\rX)}.
\end{equation}
The smallest $C > 0$ such that \eqref{eq:ineq R-bddness} is valid is then called the {\em $\cR$-bound} of $\cT$, and it is denoted by $\cR(\cT)$.
Let us recall that a closed linear operator $A \colon \rD(A) \subset \rX \to \rX$ is sectorial, denoted by $A \in \cS(\rX)$, if
\begin{enumerate}[(a)]
    \item $\overline{\rD(A)} = \overline{\mathrm{R}(A)} = \rX$, i.e., the operator is densely defined and has dense range, 
    \item $(-\infty,0) \subset \rho(A)$, and
    \item there is $M \in (0,\infty)$ with $\| t(t+A)^{-1} \|_{\cL(\rX)} \le M$ for all $t > 0$.
\end{enumerate}
For $\Sigma_\theta \coloneqq \{z \in \C \setminus \{0\} : |\arg z| < \theta\}$ denoting the sector of angle $\theta \in (0,\pi]$ in the complex plane, the spectral angle $\phi_A$ of a sectorial operator $A$ is then defined by
\begin{equation*}
    \phi_A \coloneqq \inf\left\{\phi : \rho(-A) \supset \Sigma_{\pi-\phi}, \enspace \sup_{\lambda \in \Sigma_{\pi-\phi}} \| \lambda (\lambda + A)^{-1} \| < \infty\right\}.
\end{equation*}
Note that if $A$ is a sectorial operator on a Banach space $\rX$ with spectral angle $\phi_A < \nicefrac{\pi}{2}$, then $-A$ generates a bounded analytic semigroup of angle $\nicefrac{\pi}{2} - \phi_A$.
A sectorial operator $A$ on a Banach space $\rX$ is called {\em $\cR$-sectorial} if 
\begin{equation*}
    \cR_A(0) \coloneqq \cR\left\{t(t+A)^{-1} : t > 0\right\} < \infty,
\end{equation*}
and the $\cR$-angle of an $\cR$-sectorial operator $A$ is defined by
\begin{equation*}
    \phi_A^\cR \coloneqq \inf\{\theta \in (0,\pi) : \cR_A(\pi-\theta) < \infty\}, \twhere \cR_A(\theta) \coloneqq \cR\left\{\lambda(\lambda+A)^{-1} : |\arg \lambda| \le \theta\right\}.
\end{equation*}
In what follows, we will denote the class of $\cR$-sectorial operators on a Banach space $\rX$ by $\cR \cS(\rX)$.

The importance of the concept of $\cR$-sectoriality in the context of maximal $\rL^p$-regularity becomes apparent in the following result on the characterization of maximal $\rL^p$-regularity on UMD spaces in terms of $\cR$-sectoriality.
This result is due to Weis \cite[Thm.~4.2]{Wei:01}, see also \cite[Thm.~4.4]{DHP:03}.

\begin{lem}\label{lem:char max reg via R-sect}
Consider a UMD space $\rX$ as well as $p \in (1,\infty)$, and let $A$ be a sectorial operator on $\rX$ with spectral angle $\phi_A < \nicefrac{\pi}{2}$.
Then $A$ has maximal $\rL^p$-regularity as made precise in \autoref{prop:Arendt-Bu thm} if and only if $A$ is $\cR$-sectorial with $\cR$-angle $\phi_A^\cR < \nicefrac{\pi}{2}$.
\end{lem}

Another important operator theoretic concept is the {\em boundedness of imaginary powers}.
In fact, we say that a sectorial operator $A$ on a Banach spaces has bounded imaginary powers, or $A \in \cBIP(\rX)$, if $A^{\mri s} \in \cL(\rX)$ holds for all $s \in \R$, and there is a constant $C > 0$ such that $\| A^{\mri s} \|_{\cL(\rX)} \le C$ for $|s| \le 1$.
The associated power angle $\theta_A$ is defined by $\theta_A \coloneqq \limsup_{|s| \to \infty} \frac{1}{|s|} \log |A^{\mri s}|$.
The boundedness of imaginary powers is closely related with fractional power spaces:
Consider $A \in \cS(\rX)$ and $\alpha \in (0,\infty)$.
Then the space $\rX_{A^\alpha}$ is defined by $\rX_{A^\alpha} \coloneqq (\rD(A^\alpha),\| \cdot \|_\alpha)$, where $\| x \|_\alpha \coloneqq \| x \| + \| A^\alpha x \|$.
For the following result, we refer, e.g., to \cite[Thm.~3.3.7]{PS:16}.

\begin{lem}\label{lem:frac power spaces}
For $A \in \cBIP(\rX)$, and with $[\cdot,\cdot]_\theta$ denoting the complex interpolation functor, we have $\rX_{A^\theta} \simeq [\rX,\rD(A)]_\theta$.
\end{lem}

The last concept to be introduced is the boundedness of the $\Hinfty$-calculus.
We say that a sectorial operator $A$ with spectral angle $\phi_A \in [0,\pi)$ and $\theta \in (\phi_A,\pi)$ admits a {bounded $\Hinfty$-calculus}, or $A \in \Hinfty(\rX)$, if there is a constant $C > 0$ such that
\begin{equation*}
    \| f(A) \|_{\cL(\rX)} \le C \cdot \| f \|_{\rH^\infty(\Sigma_\theta)} \tforall f \in \rH_0^\infty(\Sigma_\theta),
\end{equation*}
where $\| f \|_{\rH^\infty(\Sigma_\theta)} \coloneqq \sup\{|f(\lambda)| : |\arg \lambda | < \theta \}$, and $\rH_0^\infty(\Sigma_\theta)$ represents the bounded and holomorphic functions on $\Sigma_\theta$ such that there exist $C$, $\eps > 0$ so that 
\begin{equation*}
    |f(\lambda)| \le C \cdot \left|\frac{\lambda}{(1+\lambda)^2}\right|^\eps, \tfor \lambda \in \Sigma_\theta,
\end{equation*}
Moreover, $\phi_A^\infty$ defined by the infimum of such $\theta$ is referred to as the $\Hinfty$-angle of $A$.

The following inclusion sheds some light on the relations of the previous results.
We also refer to \cite[Sec.~4.4]{DHP:03} for more details.
In fact, we get
\begin{equation}\label{eq:rels op theoretic concepts}
    A \in \Hinfty(\rX) \subset \cBIP(\rX) \subset \cR \cS(\rX) \subset \cS(\rX), \twith \phi_A^\infty \ge \theta_A \ge \phi_A^\cR \ge \phi_A.
\end{equation}
With regard to \eqref{eq:rels op theoretic concepts} and \autoref{lem:char max reg via R-sect}, establishing the boundedness of the $\Hinfty$-calculus with an angle strictly less than $\nicefrac{\pi}{2}$ is especially sufficient to show the maximal $\rL^p$-regularity on a UMD space $\rX$.

We complete the first part of this preparatory subsection with the following result on the preservation of the $\cR$-sectoriality and the boundedness of the $\Hinfty$-calculus under similarity transforms.
This will be crucial for the decoupling argument used in \autoref{ssec:multilayered fluid-structure operator}.
For the assertion of~(a), we refer to \cite[Sec.~4.1]{DHP:03}, while the assertion of~(b) can be found in \cite[Prop.~2.11(vi)]{DHP:03}.

\begin{lem}\label{lem:sim trafos}
Consider Banach spaces $\rX$ and $\rY$ as well as $T \in \cL(\rX,\rY)$.
Let $A$ be a sectorial operator on $\rX$, and define $A_1 \coloneqq T A T^{-1}$.
\begin{enumerate}[(a)]
    \item It holds that $A \in \cR \cS(\rX)$ if and only if $A_1 \in \cR \cS(\rY)$ and $\phi_A^\cR = \phi_{A_1}^\cR$.
    \item We have $A \in \Hinfty(\rX)$ if and only if $A_1 \in \Hinfty(\rY)$ and $\phi_A^\infty = \phi_{A_1}^\infty$.
\end{enumerate}
\end{lem}

Let us proceed with the Stokes operator.
To this end, we introduce the space of solenoidal vector fields~$\rL_\sigma^q(\Omega_\rf)$ defined by $\rL_\sigma^q(\Omega_\rf) \coloneqq \overline{\{f \in \rC_\mathrm{c}^\infty(\Omega_\rf)^3 : \mdiv f = 0, \tin \Omega_\rf\}}^{\| \cdot \|_{\rL^q}}$.
Additionally invoking the gradient fields $\rG^q(\Omega_\rf) \coloneqq \{\nabla f : f \in \rL_\mathrm{loc}^q(\overline{\Omega_\rf}), \enspace \nabla f \in \rL^q(\Omega_\rf)^3\}$, we deduce from \cite[Thm.~1.4]{SS:92} that $\rL^q(\Omega_\rf)^3$ allows for the decomposition $\rL^q(\Omega_\rf)^3 = \rL_\sigma^q(\Omega_\rf) \oplus \rG^q(\Omega_\rf)$.
The resulting projection $\bP \colon \rL^q(\Omega_\rf)^3 \to \rL_\sigma^q(\Omega_\rf)$ is referred to as {\em Helmholtz projection}, and we define the Stokes operator $A_0$ with homogeneous Dirichlet boundary conditions by 
\begin{equation}\label{eq:Stokes op with hom bc}
    A_0 u \coloneqq \bP \Delta u, \tfor u \in \rD(A_0) = \rW^{2,q}(\Omega_\rf)^3 \cap \rW_0^{1,q}(\Omega_\rf)^3 \cap \rL_\sigma^q(\Omega_\rf).
\end{equation}

For the following result, we refer to \cite[Thm.~3]{NS:03}, and we observe that $\rL_\sigma^q(\Omega_\rf)$ is a UMD space.
Thus, the boundedness of the $\cH^\infty$-calculus especially implies the maximal $\rL^p$-regularity by \eqref{eq:rels op theoretic concepts} together with \autoref{lem:char max reg via R-sect}.

\begin{lem}
Consider $q \in (1,\infty)$.
For the Stokes operator $A_0$ as introduced in \eqref{eq:Stokes op with hom bc}, it holds that $-A_0 \in \cH^\infty(\rL_\sigma^q(\Omega_\rf))$ with angle $\phi_{-A_0}^\infty = 0$, and the spectral bound $s(A_0) \coloneqq \sup\{\Rep \lambda : \lambda \in \sigma(A_0)\}$ satisfies $s(A_0) < 0$.
In particular, $-A_0$ admits maximal $\rL^p$-regularity on $\rL_\sigma^q(\Omega_\rf)$.
\end{lem}

Next, we consider the damped plate part.
For this purpose, we define $\rW_0^{2,q}(\omega) \coloneqq \overline{\rC_\mathrm{c}^\infty(\omega)}^{\| \cdot \|_{\rW^{2,q}}}$ and set~$\rX_0^{\rp} \coloneqq \rW_0^{2,q}(\omega) \cap \rL_\rm^q(\omega) \times \rL_\rm^q(\omega)$.
We then define the operator matrix $A_\rp \colon \rD(A_\rp) \subset \rX_0^{\rp} \to \rX_0^{\rp}$ by
\begin{equation}\label{eq:op matrix damped plate}
    A_\rp \coloneqq \begin{pmatrix}
        0 & \Id\\
        -P_\rm \Delta_\rp^2 & \Delta_\rp
    \end{pmatrix}, \twhere \rD(A_\rp) = \left(\rW^{4,q}(\omega) \cap \rW_0^{2,q}(\omega) \cap \rL_\rm^q(\omega)\right) \times \left(\rW_0^{2,q}(\omega) \cap \rL_\rm^q(\omega)\right).
\end{equation}
Here $\Delta_\rp^2$ and $\Delta_\rp$ represent the $\rL^q(\omega)$-realization of the in-plane bi-Laplacian $\Delta_s^2$ and Laplacian $\Delta_s$ as introduced in \autoref{sec:intro} and subject to the clamped boundary conditions.
The result below can be found in \cite[Thm.~4.2]{MT:21}, see also \cite[Thm.~5.1]{DS:15} for the underlying solvability of the damped plate problem on the spaces without spatial average zero.

\begin{lem}
Let $q \in (1,\infty)$, and recall $A_\rp$ from \eqref{eq:op matrix damped plate}.
Then there is $\mu \ge 0$ such that $-A_\rp + \mu$ is $\cR$-sectorial on $\rX_0^\rp$ with angle $\phi_{-A_\rp + \mu}^\cR < \nicefrac{\pi}{2}$.
In particular, $-A_\rp + \mu$ has maximal $\rL^p$-regularity on $\rX_0^\rp$.
\end{lem}

\subsection{The thick structural layer}\label{ssec:thick structure}
\ 

In this subsection, we investigate the thick structural layer, i.e., we study the block that corresponds to $d_1$ and $d_2$.
More precisely, we establish the boundedness of the $\Hinfty$-calculus of the associated operator matrix and reveal that its spectral bound is negative.
To the best of our knowledge, the boundedness of the $\Hinfty$-calculus of this block operator matrix is new, and it is of independent interest.

First, we make precise the elasticity operator.
For $f \in \rW^{2,q}(\Omega_\rs)^3$, and $\mu_\rs$, $\lambda_\rs$ satisfying \eqref{eq:Lame coeffs}, we define
\begin{equation*}
    \cL f \coloneqq \mu_\rs \mdiv\bigl(\nabla f + (\nabla f)^\top\bigr) + \lambda_\rs \mdiv\bigl((\mdiv f) \Id_3\bigr).
\end{equation*}
Likewise, we introduce the elasticity operator $L_0$ subject to homogeneous boundary conditions.
Indeed, the operator $L_0$ is the $\rL^q(\Omega_\rs)^3$-realization of the differential operator $\cL$ subject to homogeneous Dirichlet boundary conditions, i.e., we set
\begin{equation}\label{eq:Lame op}
    L_0 f \coloneqq \mu_\rs \mdiv\bigl(\nabla f + (\nabla f)^\top\bigr) + \lambda_\rs \mdiv\bigl((\mdiv f) \Id_3\bigr), \tfor f \in \rD(L_0) = \rW^{2,q}(\Omega_\rs)^3 \cap \rW_0^{1,q}(\Omega_\rs)^3. 
\end{equation}
The result below on functional analytic and spectral properties of the Lam\'e operator from \eqref{eq:Lame op} is somewhat classical.
The assertion on the spectral bound is implied by the compactness of the resolvent of $L_0$ together with a simple testing argument of the eigenvalue problem in $\rL^2$, since the $q$-independent spectrum of $-L_0$ consists of discrete eigenvalues.
The second part of the assertion can be deduced from the strong normal ellipticity of the Lam\'e operator thanks to the assumptions on the Lam\'e coefficients, see also \cite[Sec.~6.2]{PS:16} and in particular \cite[Thm.~6.2.11]{PS:16}, which yields the required ellipticity and validity of the Lopatinskii-Shapiro condition so that the assertion follows from \cite[Thm.~2.3]{DDHPV:04}.

\begin{lem}\label{lem:props of the Lame op}
Let $q \in (1,\infty)$, and assume that $\mu_\rs$ and $\lambda_\rs$ satisfy \eqref{eq:Lame coeffs}.
Then it holds that $s(L_0) < 0$ and $-L_0 \in \Hinfty(\rL^q(\Omega_\rs)^3)$ with angle $\phi_{-L_0}^\infty < \nicefrac{\pi}{2}$.
\end{lem}

In the following, we shall tacitly assume that the Lam\'e coefficients $\mu_\rs$ and $\lambda_\rs$ satisfy \eqref{eq:Lame coeffs} are as made precise in \autoref{lem:props of the Lame op} without further notice. We now introduce an operator matrix in order to rewrite the thick structural layer problem as a first order problem on the space $\rX_0^\rs \coloneqq \rW_0^{1,q}(\Omega_\rs)^3 \times \rL^q(\Omega_\rs)^3$.
The difficulty in the analysis of the thick structural layer is that only the sum $d + \delta \del_t d$ enjoys $\rW^{2,q}$-regularity.

\begin{rem}\label{rem:naive op matrix thick structural layer}
At first sight, it may be tempting to consider the operator $E_\rs \colon \rD(E_\rs) \subset \rX_0^\rs \to \rX_0^\rs$ given by
\begin{equation*}
    E_\rs \binom{d_1}{d_2} = \binom{d_2}{L_0 d_1 + \delta L_0 d_2}, \tfor \binom{d_1}{d_2} \in \rD(E_\rs) = \rD(L_0) \times \rD(L_0).
\end{equation*}
However, note that the resulting operator is {\em not} closed, see \cite[Sec.~2.1]{CCD:08}.
\end{rem}

Instead of the operator matrix from \autoref{rem:naive op matrix thick structural layer}, we consider $A_\rs \colon \rD(A_\rs) \subset \rX_0^\rs \to \rX_0^\rs$ associated with the thick structural layer taking the shape
\begin{equation}\label{eq:op matrix thick structure}
    A_\rs \coloneq \begin{pmatrix}
        0 & \Id\\
        L_0 & \delta L_0
    \end{pmatrix}, \twith \rD(A_\rs) = \bigl\{(d_1,d_2) \in \rW_0^{1,q}(\Omega_\rs)^3 \times \rW_0^{1,q}(\Omega_\rs)^3 : d_1 + \delta d_2 \in \rD(L_0)\bigr\},
\end{equation}
and where the action of the operator matrix $A_\rs$ is defined by
\begin{equation*}
    A_\rs \binom{d_1}{d_2} = \binom{d_2}{L_0(d_1 + \delta d_2)}, \tfor \binom{d_1}{d_2} \in \rD(A_\rs).
\end{equation*}
By \autoref{lem:props of the Lame op} as well as \eqref{eq:rels op theoretic concepts}, we especially have $-L_0 \in \cBIP(\rL^q(\Omega_\rs)^3)$ with power angle $\theta_{-L_0} < \nicefrac{\pi}{2}$, so by \autoref{lem:frac power spaces}, we get
\begin{equation}\label{eq:frac power domain Lame}
    \rW_0^{1,q}(\Omega_\rs)^3 = \bigl[\rL^q(\Omega_\rs)^3,\rW^{2,q}(\Omega)^3 \cap \rW_0^{1,q}(\Omega_\rs)^3\bigr]_{\nicefrac{1}{2}} = \bigl[\rL^q(\Omega_\rs)^3,\rD(-L_0)\bigr]_{\nicefrac{1}{2}} = \rD((-L_0)^{\nicefrac{1}{2}}).
\end{equation}

Below, we discuss functional analytic as well as spectral properties of the operator matrix $A_\rs$.
Note that the assumption $\delta > 0$ is indispensable for obtaining the following result on the boundedness of the $\Hinfty$-calculus and thus also the maximal regularity of the operator matrix associated with the thick structural layer.

\begin{prop}\label{prop:bdd Hinfty & spectral bound thick structural layer}
Let $q \in (1,\infty)$, and consider $A_\rs$ as in \eqref{eq:op matrix thick structure}.
Then the following assertions are valid.
\begin{enumerate}[(a)]
    \item It holds that $s(A_\rs) < 0$.
    \item We have $-A_\rs \in \Hinfty(\rX_0^\rs)$ with angle $\phi_{-A_\rs}^\infty < \nicefrac{\pi}{2}$.
    \item The operator $-A_\rs$ is $\cR$-sectorial on $\rX_0^\rs$ with $\cR$-angle $\phi_{-A_\rs}^{\cR} < \nicefrac{\pi}{2}$.
    In particular, $-A_\rs$ has maximal $\rL^p$-regularity on $\rX_0^\rs$, and $A_\rs$ generates a bounded analytic semigroup $(\mre^{tA_\rs })_{t \ge 0}$ on $\rX_0^\rs$.
\end{enumerate}
\end{prop}

\begin{proof}
The first assertion of~(c) immediately follows from~(b) upon recalling \eqref{eq:rels op theoretic concepts}.
The second assertion then is a consequence of the first one thanks to the UMD property of the space $\rX_0^\rs$ by $q \in (1,\infty)$ and the characterization of maximal $\rL^p$-regularity in terms of $\cR$-sectoriality.
The last part of the assertion of~(c) is implied by the former ones and~(a) in a standard way, see, e.g., \cite[Sec.~3.5]{PS:16}.

Hence, it remains to show~(a) and~(b).
In fact, we will first prove that $-A_\rs + \mu \in \Hinfty(\rX_0^\rs)$ for some~$\mu \ge 0$ and then investigate the spectral bound.
Conceptually, our proof is inspired by the considerations in \cite[Sec.~2]{CCD:08}.
However, let us emphasize that in \cite{CCD:08}, the authors are merely interested in the sectoriality, while we strive for the boundedness of the $\Hinfty$-calculus here.
At this stage, note that perturbation theory for the latter concept is significantly more involved than for sectoriality.

As an auxiliary tool, we introduce the operator matrix $B_\rs \colon \rD(B_\rs) \subset \rX_0^\rs \to \rX_0^\rs$, defined by
\begin{equation*}
    B_\rs \coloneqq A_\rs - \begin{pmatrix}
        0 & 0\\
        0 & \delta^{-1} \Id
    \end{pmatrix} = \begin{pmatrix}
        0 & \Id\\
        L_0 & \delta L_0 - \delta^{-1} \Id
    \end{pmatrix}, \twith \rD(B_\rs) = \rD(A_\rs).
\end{equation*}
Moreover, we invoke the isomorphism $S \colon \rX_0^\rs \to \rX_0^\rs$ as well as its inverse  given by
\begin{equation*}
    S = \begin{pmatrix}
        \Id & 0\\
        \delta^{-1} & \Id
    \end{pmatrix} \tand S^{-1} = \begin{pmatrix}
        \Id & 0\\
        -\delta^{-1} & \Id
    \end{pmatrix},
\end{equation*}
respectively.
We then introduce the operator matrix $C_\rs \colon \rD(C_\rs) \subset \rX_0^\rs \to \rX_0^\rs$ taking the shape
\begin{equation*}
    C_\rs = S B_\rs S^{-1} = \begin{pmatrix}
        -\delta^{-1} \Id & \Id\\
        0 & \delta L_0
    \end{pmatrix}, \twhere \rD(C_\rs) = S \rD(B_\rs) = \rW_0^{1,q}(\Omega_\rs)^3 \times \left(\rW^{2,q}(\Omega)^3 \cap \rW_0^{1,q}(\Omega_\rs)^3\right).
\end{equation*}
Next, we invoke the diagonal part $\Tilde{C}_\rs \colon \rD(\Tilde{C}_\rs) \subset \rX_0^\rs \to \rX_0^\rs$ of $C_\rs$, i.e.,
\begin{equation}
    \Tilde{C}_\rs \coloneqq \diag(-\delta^{-1} \Id,\delta L_0), \twith \rD(\Tilde{C}_\rs) = \rD(C_\rs).
\end{equation}
From \autoref{lem:props of the Lame op}, we immediately infer that $-\Tilde{C}_\rs \in \Hinfty(\rX_0^\rs)$ with $\phi_{-\Tilde{C}_\rs}^\infty < \nicefrac{\pi}{2}$.
The next goal is to transfer this property from $\Tilde{C}_\rs$ to $C_\rs$.
To this end, observe that $C_\rs = \Tilde{C}_\rs + \Tilde{P}_\rs$, where
\begin{equation*}
    \Tilde{P}_\rs = \begin{pmatrix}
        0 & \Id\\
        0 & 0
    \end{pmatrix}.
\end{equation*}
For this purpose, let $(d_1,d_2) \in \rD(\Tilde{C}_\rs)$.
Making use of \eqref{eq:frac power domain Lame} as well as $0 \in \rho(-L_0)$, which is implied by \autoref{lem:props of the Lame op}, together with the shape of $\Tilde{C}_\rs$, we find that
\begin{equation*}
    \left\| -\Tilde{P}_\rs \binom{d_1}{d_2} \right\|_{\rX_0^\rs} = \| d_2 \|_{\rW_0^{1,q}(\Omega_\rs)} = \left\| (-L_0)^{\nicefrac{1}{2}} d_2 \right\|_{\rL^q(\Omega_\rs)} \le C \cdot \left\| \Tilde{C}_\rs^{\nicefrac{1}{2}} \binom{d_1}{d_2} \right\|_{\rX_0^\rs}.
\end{equation*}
In other words, $-C_\rs$ is bounded with respect to a fractional power of $-\Tilde{C}_\rs$.
It then follows from \cite[Prop.~13.1]{KW:04} that there exists $\mu \ge 0$ such that $-C_\rs + \mu = -\Tilde{C}_\rs - \Tilde{P}_\rs + \mu \in \Hinfty(\rX_0^\rs)$ with $\phi_{-C_\rs + \mu}^\infty < \nicefrac{\pi}{2}$.
Observe that the choice $\mu = 0$ is possible, since $C_\rs$ is invertible.
Indeed, by $0 \in \rho(L_0)$, we have
\begin{equation*}
    C_\rs^{-1} = \begin{pmatrix}
        -\delta \Id & L_0^{-1}\\
        0 & \delta^{-1} L_0^{-1}
    \end{pmatrix},
\end{equation*}
revealing that $0 \in \rho(C_\rs)$, so $-C_\rs \in \Hinfty(\rX_0^\rs)$ with $\phi_{-C_\rs}^\infty < \nicefrac{\pi}{2}$.

Furthermore, as the $\Hinfty$-calculus is preserved under isomorphisms, see \autoref{lem:sim trafos}(b), we infer that $-B_\rs \in \Hinfty(\rX_0^\rs)$ with $\phi_{-B_\rs}^\infty < \nicefrac{\pi}{2}$.
Next, note that $A_\rs = B_\rs + P_\rs$, with $P_\rs = \diag(0,\delta^{-1} \Id)$.
Let now $(d_1,d_2) \in \rD(B_\rs) = \rD(A_\rs)$.
Then we find that
\begin{equation*}
    \left\| P_\rs \binom{d_1}{d_2} \right\|_{\rX_0^\rs} = \| \delta^{-1} d_2 \|_{\rL^q(\Omega_\rs)} \le C \cdot \left\| \binom{d_1}{d_2} \right\|_{\rX_0^\rs}.
\end{equation*}
In other words, $-A_\rs$ is a bounded perturbation of $-B_\rs$, so, a priori up to a shift, the boundedness of the $\Hinfty$-calculus with angle strictly less than $\nicefrac{\pi}{2}$ transfers to $-A_\rs$.

In order to handle the shift, and in order to show the assertion of~(a), we now verify that $0 \in \rho(A_\rs)$.
Again, by virtue of $0 \in \rho(L_0)$, we find that $A_\rs^{-1}$ admits the explicit representation
\begin{equation*}
    A_\rs^{-1} = \begin{pmatrix}
        -\delta & L_0^{-1}\\
        \Id & 0
    \end{pmatrix},
\end{equation*}
revealing the desired invertibility and thereby completing the proof.
\end{proof}

\subsection{Lifting arguments}\label{ssec:lifting args}
\ 

After recalling properties of the Stokes operator with homogeneous boundary conditions and the thin structural layer in \autoref{ssec:prelim cons}, and with the understanding of the thick structure gained in \autoref{ssec:thick structure}, we are now in the position to tackle the coupling conditions by means of lifting arguments.
Let us start with the lifting associated with the continuity of velocities at the interface, i.e., $u = P_\rm(\eta_2) \mre_3$ on $\Gamma_0$.
At this stage, we observe that the underlying lifting procedure has already been carried out in \cite[Sec.~4.3]{MT:21}.
However, for convenience, we recall the main steps here.

To this end, take into account the stationary Stokes problem with the corresponding boundary conditions.
Let us recall that $\del \Omega_\rf = \Gamma_0 \cup \Gamma_\rf$ is especially of class $\rC^3$, and that $\Gamma_0 = \{(s,0) : s \in \omega\}$.
We consider
\begin{equation}\label{eq:stat Stokes probl with inhom bc}
    \left\{
    \begin{aligned}
        -\mdiv \sigma_\rf(w,\psi)
        &= f, \enspace \mdiv w = 0, &&\tin \Omega_\rf,\\
        w
        &= b, &&\ton \Gamma_0,\\
        w
        &= 0, &&\ton \Gamma_\rf,\\
        \int_{\Omega_\rf} \psi \srd y
        &= 0.
    \end{aligned}
    \right.
\end{equation}

The following result is classical, see for example \cite[Prop.~2.3]{Tem:79}.

\begin{lem}\label{lem:sol stat Stokes probl with inhom bc}
Let $q \in (1,\infty)$, $f \in \rL^q(\Omega_\rf)^3$ and $b \in \rW_0^{2,q}(\omega)^3 \cap \rL_\rm^q(\omega)^3$.
Then there exists a unique solution~$(w,\psi) \in \rW^{2,q}(\Omega_\rf)^3 \times \rW^{1,q}(\Omega_\rf) \cap \rL_\rm^q(\Omega_\rf)$ to \eqref{eq:stat Stokes probl with inhom bc}.
\end{lem}

Thanks to \autoref{lem:sol stat Stokes probl with inhom bc}, we may introduce the lifting operators $D_\fl \in \cL(\rW_0^{2,q}(\omega)^3 \cap \rL_\rm^q(\omega)^3,\rW^{2,q}(\Omega_\rf)^3)$ as well as $D_\pr \in \cL(\rW_0^{2,q}(\omega)^3 \cap \rL_\rm^q(\omega)^3,\rW^{1,q}(\Omega_\rf) \cap \rL_\rm^q(\Omega_\rf))$ defined by
\begin{equation}\label{eq:lifting ops fluid & pressure}
    D_\fl b \coloneqq w \tand D_\pr b \coloneqq \psi,
\end{equation}
where $(v,\psi)$ is the unique solution to the stationary Stokes problem with inhomogeneous boundary conditions \eqref{eq:stat Stokes probl with inhom bc} resulting from \autoref{lem:sol stat Stokes probl with inhom bc}.

Next, we elaborate on the handling of the pressure by considering (weak) Neumann problems.
In fact, we take into account the Neumann problem 
\begin{equation}\label{eq:Neumann probl}
    \Delta \varphi = 0, \tin \Omega_\rf, \enspace \del_\nu \varphi = c, \ton \del \Omega_\rf, \enspace \int_{\Omega_\rf} \varphi \srd y = 0,
\end{equation}
and we denote by $N$ the solution operator to the Neumann problem \eqref{eq:Neumann probl}, so 
\begin{equation}\label{eq:lifting op N}
    N c \coloneqq \varphi.
\end{equation}
Observe that~$\int_{\del \Omega_\rf} c \srd S = 0$ has to hold by the divergence theorem.
We then define the space $\rW_\rm^{-\nicefrac{1}{q},q}(\del \Omega_\rf)$ by $\rW_\rm^{-\nicefrac{1}{q},q}(\del \Omega_\rf) \coloneqq \bigl\{f \in \rW^{-\nicefrac{1}{q},q}(\del \Omega_\rf) : \langle f,1 \rangle_{\rW^{-\nicefrac{1}{q},q},\rW^{1-\nicefrac{1}{q'},q'}} = 0\bigr\}$, for $\nicefrac{1}{q} + \nicefrac{1}{q'} = 1$ and note that the solution operator $N$ possesses the mapping properties
\begin{equation}\label{eq:mapping props of Neumann op}
    \begin{aligned}
        N 
        &\in \cL(\rW^{1-\nicefrac{1}{q},q}(\del \Omega_\rf) \cap \rL_\rm^q(\del \Omega_\rf),\rW^{2,q}(\Omega_\rf) \cap \rL_\rm^q(\Omega_\rf)),\\
        N 
        &\in \cL(\rW_\rm^{-\nicefrac{1}{q},q}(\del \Omega_\rf),\rW^{1,q}(\Omega_\rf) \cap \rL_\rm^q(\Omega_\rf)), \tand\\
        N
        &\in \cL(\rL_\rm^q(\del \Omega_\rf),\rW^{1+\nicefrac{1}{q}-\eps,q}(\Omega_\rf) \cap \rL_\rm^q(\Omega_\rf)), \tforall \eps > 0,
    \end{aligned}
\end{equation}
see for example \cite[Thm.~4.2 and~4.3]{LM:62}.
In this context, we also invoke the adjusted lifting operator 
\begin{equation}\label{eq:lifting op N_1}
    N_1 c_1 = N c, \twhere c(y) = 
    \left\{
    \begin{aligned}
        c_1(s), &\tif y = (s,0) \in \Gamma_0,\\
        0, &\tif y \in \Gamma_\rf.
    \end{aligned}
    \right.
\end{equation}

Concerning the thin structural layer, the last ingredient is the weak Neumann problem that corresponds to the Helmholtz projection.
Namely, we study the weak solution to the Neumann problem 
\begin{equation}\label{eq:weak Neumann probl}
    \Delta \varphi = \mdiv f, \tin \Omega_\rf, \enspace \del_\nu \varphi = f \cdot \nu, \ton \del \Omega_\rf,
\end{equation}
i.e., $\varphi \in \rW^{1,q}(\Omega_\rf)$ solves $\int_{\Omega_\rf} \nabla \varphi \cdot \nabla \psi \srd y = \int_{\Omega_\rf} f \cdot \nabla \psi \srd y$, for $\psi \in \rW^{1,q'}(\Omega_\rf)$ and $\nicefrac{1}{q} + \nicefrac{1}{q'} = 1$.
For the solvability of this problem, which is related to the existence of the Helmholtz projection, we refer again to \cite{SS:92}.
For $\varphi$ solving \eqref{eq:weak Neumann probl}, we define the solution operator $N_2 \in \cL(\rL^q(\Omega_\rf)^3,\rW^{1,q}(\Omega_\rf) \cap \rL_\rm^q(\Omega_\rf))$ by
\begin{equation}\label{eq:lifting op N_2}
    N_2 f \coloneqq \varphi.
\end{equation}

After addressing the coupling condition on the equality of velocities, we now tackle the second kinematic coupling condition, namely the continuity of displacement, i.e., $d_1 = P_\rm(\eta_1) \mre_3$ on $\Gamma_0$.
We recall that $\del \Omega_\rs = \Gamma_0 \cup \Gamma_\rs$ is of class $\rC^2$ by assumption.
For~$\lambda \in \C$, we then consider the stationary problem
\begin{equation}\label{eq:Lame probl with inhom bc}
    \left\{
    \begin{aligned}
        \lambda d - \cL d
        &= 0, &&\tin \Omega_\rs,\\
        d
        &= b', &&\ton \Gamma_0,\\
        d
        &= 0, &&\ton \Gamma_\rs.
    \end{aligned}
    \right.
\end{equation}
In the lemma below, we discuss the solvability of \eqref{eq:Lame probl with inhom bc}.
The discussion preceding \autoref{lem:props of the Lame op} together with the fact that $0 \in \rho(L_0)$ by \autoref{lem:props of the Lame op} reveals that we are in the framework of \cite[Thm.~2.1]{DHP:07}.
The latter implies that for all $b' \in \rW_0^{2,q}(\omega)^3 \cap \rL_\rm^q(\omega)^3$, there exists a unique solution $d \in \rW^{2,q}(\Omega_\rs)^3$ to~\eqref{eq:Lame probl with inhom bc}.

\begin{lem}\label{lem:solvability stat Lame probl}
Let $q \in (1,\infty)$ and $b' \in \rW_0^{2,q}(\omega)^3 \cap \rL_\rm^q(\omega)^3$.
Then for every $\lambda \in \C$ with $\Rep \lambda \ge 0$, there exists a unique solution $d \in \rW^{2,q}(\Omega_\rs)^3$ to \eqref{eq:Lame probl with inhom bc}.
\end{lem}

\autoref{lem:solvability stat Lame probl} allows us to define the lifting operator $D_\rs \colon \rW_0^{2,q}(\omega)^3 \cap \rL_\rm^q(\omega)^3 \to \rW^{2,q}(\Omega_\rs)^3$ by
\begin{equation}\label{eq:lifting op Lame}
    D_\rs b' \coloneqq d,
\end{equation}
where $d$ is the solution to \eqref{eq:Lame probl with inhom bc} in the case $\lambda = 0$.
Let us observe that the aforementioned operator can also be extended to a continuous operator $\Tilde{D}_\rs \colon \rL_\rm^q(\omega)^3 \to \rW^{s,q}(\Omega_\rs)^3$ for all $s \in (0,\nicefrac{1}{q})$.
For the existence of a Dirichlet operator in a similar context, we also refer to \cite[Sec.~7.1]{BBH:26}.
By a slight abuse of notation, we still denote the associated operator by $D_\rs$ for simplicity.

\subsection{The multilayered fluid-structure operator}\label{ssec:multilayered fluid-structure operator}
\ 

Thanks to the preparation in \autoref{ssec:lifting args}, we may now start to reformulate the linearized problem~\eqref{eq:lin per multilayered interaction probl} based on the lifting operators $D_\fl$, $D_\pr$ and $D_\rs$ as introduced in \eqref{eq:lifting ops fluid & pressure} and \eqref{eq:lifting op Lame} as well as the Neumann operators $N$, $N_1$ and $N_2$ from \eqref{eq:lifting op N}, \eqref{eq:lifting op N_1} and \eqref{eq:lifting op N_2}.
Let us note that we reformulate the {\em initial value problem} instead of the time-periodic problem.
By our methodology, we require a good understanding of the initial value problem, because we reduce the linear time-periodic problem to the initial value problem and a spectral condition.
In order to shorten notation, for the Piola-Kirchhoff stress tensor $\sigma_\rs$ from \eqref{eq:Piola-Kirchhoff stress tensor}, we also introduce the operator $T_\rs$ defined by
\begin{equation}
    T_\rs(d_1,d_2) \coloneqq P_\rm \left(\left.(\sigma_\rs(d_1,d_2) \mre_3)\right|_{\Gamma_0} \cdot \mre_3\right).
\end{equation}

The result on the equivalent reformulation of \eqref{eq:lin per multilayered interaction probl} now reads as follows.

\begin{prop}\label{prop:equiv reform lin probl}
Let $p$, $q \in (1,\infty)$, and consider
\begin{equation*}
    \begin{aligned}
        u
        &\in \rW^{1,p}(0,T;\rL^q(\Omega_\rf)^3) \cap \rL^p(0,T;\rW^{2,q}(\Omega_\rf)^3), \enspace \pi \in \rL^p(0,T;\rW^{1,q}(\Omega_\rf) \cap \rL_\rm^q(\Omega_\rf)),\\
        \eta_1 
        &\in \rW^{2,p}(0,T;\rL^q(\omega)) \cap \rW^{1,p}(0,T;\rW^{2,q}(\omega)) \cap \rL^p(0,T;\rW^{4,q}(\omega)),\\
        \eta_2 
        &\in \rW^{1,p}(0,T;\rL^q(\omega)) \cap \rL^p(0,T;\rW^{2,q}(\omega)),\\
        d_1 
        &\in \rW^{1,p}(0,T;\rW^{1,q}(\Omega_\rs)^3), \enspace d_2 \in \rW^{1,p}(0,T;\rL^q(\Omega_\rs)^3) \cap \rL^p(0,T;\rW^{1,q}(\Omega_\rs)^3) \tand\\
        d_1 + \delta d_2 
        &\in \rL^p(0,T;\rW^{2,q}(\Omega_\rs)^3). 
    \end{aligned}
\end{equation*}
Then $(u,\pi,\eta_1,\eta_2,d_1,d_2)$ is a solution to \eqref{eq:lin per multilayered interaction probl}, with the periodicity conditions replaced by suitable initial conditions $u(0) = u_0$, $\eta_1(0) = \eta_{1,0}$, $\eta_2(0) = \eta_{2,0}$, $d_1(0) = d_{1,0}$ and $d_2(0) = d_{2,0}$, if and only if
\begin{equation*}
    \left\{
    \begin{aligned}
        \bP u'
        &= A_0 \bP(u - D_\fl \eta_2) + \bP f, &&\tin (0,T),\\
        \eta_1'
        &= \eta_2, &&\tin (0,T),\\
        (\Id + \gamma_\rm N_1) \eta_2' + P_\rm \Delta_\rp^2 \eta_1 - \Delta_\rp \eta_2
        &= \gamma_\rm N(\Delta \bP u \cdot \nu) + T_\rs(d_1,d_2) + P_\rm g + \gamma_\rm N_2 f, &&\tin (0,T),\\
        \int_\omega \eta_1(\cdot,s) \srd s 
        &= 0, &&\tin (0,T),\\
        d_1'
        &= d_2, &&\tin (0,T),\\
        d_2'
        &= L_0\bigl(d_1 - D_\rs (P_\rm(\eta_1) \mre_3) + \delta (d_2 - D_\rs (P_\rm(\eta_2) \mre_3))\bigr) + h, &&\tin (0,T),\\
        (\Id - \bP)u
        &= (\Id - \bP)D_\fl \eta_2,\\
        \pi
        &= N(\Delta \bP u \cdot \nu) - N_1 \eta_2' + N_2 f,\\
        (\bP u,\eta_1,\eta_2,d_1,d_2)(0)
        &= (\bP u_0,\eta_{1,0},\eta_{2,0},d_{1,0},d_{2,0}).
    \end{aligned}
    \right.
\end{equation*}
\end{prop}

\begin{proof}
For the procedure to reformulate the Stokes equation as well as the first order system accounting for the thin structural layer, we refer to \cite[Prop.~4.4]{MT:21}.
Thus, it remains to handle the thick structural layer in the rest of the proof.

For $b' = P_\rm(\eta_1) \mre_3$, the part associated with the thick structural layer in the aforementioned initial value problem-analogue of \eqref{eq:lin per multilayered interaction probl} can be written as
\begin{equation}\label{eq:lin probl thick structural layer}
    \left\{
    \begin{aligned}
        \del_t d_1 
        &= d_2, &&\tin (0,T) \times \Omega_\rs,\\
        \del_t d_2 - \mdiv \sigma_\rs(d_1,d_2)
        &= h, &&\tin (0,T) \times \Omega_\rs,\\
        d_1
        &= b', &&\ton (0,T) \times \Gamma_0,\\
        d_1
        &= 0, &&\ton (0,T) \times \Gamma_\rs,\\
        d_1(0)
        &= d_{1,0}, \enspace d_2(0) = d_{2,0}, &&\tin \Omega_\rs.
    \end{aligned}
    \right.
\end{equation}
In order to address the inhomogeneous boundary condition in \eqref{eq:lin probl thick structural layer}$_3$, we invoke the lifting operator $D_\rs$ from \eqref{eq:lifting op Lame}.
We also set $\td_1 \coloneqq d_1 - D_\rs b'$, $\td_2 \coloneqq d_2 - D_\rs (\del_t b')$, $\td_{1,0} \coloneqq d_{1,0} - (D_\rs b')(0)$ as well as~$\td_{2,0} \coloneqq d_{2,0} - (D_\rs (\del_t b'))(0)$.
It then follows that 
\begin{equation*}
    \del_t \td_1 = \del_t d_1 - \del_t(D_\rs b') = d_2 - \del_t (D_\rs b') = \td_2 + D_\rs(\del_t b') - \del_t (D_\rs b'),
\end{equation*}
which is in turn equivalent with
\begin{equation*}
    \del_t d_1 = d_2.
\end{equation*}
On the other hand, by construction of the lifting operator $D_\rs$ as made precise in \eqref{eq:Lame probl with inhom bc} and \autoref{lem:solvability stat Lame probl}, we find that $\cL(D_\rs b' + \delta D_\rs(\del_t b')) = 0$.
As $\td_1$ and $\td_2$ further satisfy homogeneous Dirichlet boundary conditions by construction, we get $\mdiv \sigma_\rs(\td_1,\td_2) = L_0(\td_1 + \delta \td_2)$.
Thus, we obtain
\begin{equation*}
    \begin{aligned}
        \del_t \td_2 
        &= \del_t d_2 - \del_t (D_\rs(\del_t b'))\\
        &= \mdiv \sigma_\rs(d_1,d_2) + h - \del_t(D_\rs(\del_t b'))\\
        &= \mdiv \sigma_\rs(\td_1 + D_\rs b',\td_2 + D_\rs(\del_t b')) + h - \del_t(D_\rs(\del_t b'))\\
        &= \mdiv \sigma_\rs(\td_1,\td_2) + \cL(D_\rs b' + \delta D_\rs (\del_t b')) + h - \del_t(D_\rs(\del_t b'))\\
        &= L_0(\td_1 + \delta \td_2) + h - \del_t(D_\rs(\del_t b')).
    \end{aligned}
\end{equation*}
Rearranging the terms, and plugging in $b' = P_\rm(\eta_1) \mre_3$ and $\del_t b' = P_\rm(\eta_2) \mre_3$, we find that the latter expression is equivalent with
\begin{equation*}
    \del_t d_2 = L_0(d_1 - D_\rs(P_\rm(\eta_1) \mre_3) + \delta (d_2 - D_\rs(P_\rm(\eta_2) \mre_3)) + h.
\end{equation*}
This establishes the desired equivalence and thus completes the proof.
\end{proof}

Before reformulating the linearized problem in operator form, we elaborate on the so-called {\em added mass operator} $M_s$ defined by $M_s \coloneqq \Id + \gamma_\rm N_1$, where the lifting operator $N_1$ has been made precise in \eqref{eq:lifting op N_1}.
For the result below, see \cite[Lemma~4.5]{MT:21}.

\begin{lem}\label{lem:props of added mass op}
For the added mass operator $M_s = \Id + \gamma_\rm N_1 \in \cL(\rL_\rm^q(\omega))$, we have that
\begin{enumerate}[(a)]
    \item $M_s$ is an automorphism in $\rW_\rm^{s,q}(\omega)$ for all $s \in [0,1)$,
    \item $M_s - \Id \in \cL(\rL_\rm^q(\omega),\rW_\rm^{s,q}(\omega))$ for all $s \in [0,1)$, and
    \item the operator $M_s - \Id$ is compact on $\rL_\rm^q(\omega)$.
\end{enumerate}
\end{lem}

For brevity, we also introduce
\begin{equation}\label{eq:def of map K}
    \cK(f) \coloneqq P_\rm\Bigl(\bigl(\mu_\rs\bigl(\nabla f + (\nabla f)^\top\bigr) + \lambda_\rs \mdiv(f) \Id\bigr) \mre_3 \left. \right|_{\Gamma_0} \cdot \mre_3 \Bigr)
\end{equation}
for sufficiently regular $f$.
Consequently, we have $T_\rs(d_1,d_2) = \cK(d_1 + \delta d_2)$.

We now define the {\em multilayered fluid-structure operator} $A_\mfs \colon \rD(A_\mfs) \subset \rX_0 \to \rX_0$ on the ground space
\begin{equation}\label{eq:ground space multilayered}
    \begin{aligned}
        \rX_0 = \Bigl\{(\bP u,\eta_1,\eta_2,d_1,d_2) 
        &\in \rL_\sigma^q(\Omega_\rf) \times \bigl(\rW_0^{2,q}(\omega) \cap \rL_\rm^q(\omega)\bigr) \times \rL_\rm^q(\omega) \times \rW^{1,q}(\Omega_\rs)^3 \times \rL^q(\Omega_\rs)^3 :\\
        &\qquad d_1 - D_\rs(P_\rm(\eta_1) \mre_3) \in \rW_0^{1,q}(\Omega_\rs)^3\Bigr\}
    \end{aligned}
\end{equation}
by
\begin{equation}\label{eq:multilayered fluid-structure op}
    A_\mfs \coloneqq \begin{pmatrix}
            A_0 & 0 & -A_0 \bP D_\fl & 0 & 0\\
            0 & 0 & \Id & 0 & 0\\
            M_s^{-1} \gamma_\rm N \Delta(\cdot) \cdot \nu & -M_s^{-1} P_\rm \Delta_\rp^2 & M_\rs^{-1} \Delta_\rp & M_s^{-1}\cK & \delta M_s^{-1}\cK\\
            0 & 0 & 0 & 0 & \Id\\
            0 & -L_0 D_\rs (P_\rm (\cdot) \mre_3) & -\delta L_0 D_\rs (P_\rm (\cdot) \mre_3) & L_0 & \delta L_0
        \end{pmatrix}
\end{equation}
with domain
\begin{equation}\label{eq:dom multilayered fluid-structure op}
    \begin{aligned}
        \rD(A_\mfs) = \bigl\{&(\bP u,\eta_1,\eta_2,d_1,d_2) \in \rW^{2,q}(\Omega_\rf)^3 \cap \rL_\sigma^q(\Omega_\rf) \times \rD(A_\rp) \times \rW^{1,q}(\Omega_\rs)^3 \times \rW^{1,q}(\Omega_\rs)^3 :\\
        &\enspace \bP u - \bP D_\fl \eta_2 \in \rD(A_0) \tand (d_1 - D_\rs ((P_\rm \eta_1)\mre_3),d_2 - D_\rs ((P_\rm \eta_2)\mre_3)) \in \rD(A_\rs)\bigr\}.
    \end{aligned}
\end{equation}

Thanks to \autoref{prop:equiv reform lin probl}, the above operator $A_\mfs$ defined in \eqref{eq:multilayered fluid-structure op} allows us to reformulate the initial value problem-analogue of \eqref{eq:lin per multilayered interaction probl} in the following compact way.
For $\overline{g} \coloneqq M_s^{-1} P_\rm g + M_s^{-1} \gamma_\rm N_2 f$, we get
\begin{equation}\label{eq:reform in op form multilayered fs op}
    \begin{aligned}
        \frac{\srd}{\srd t} \begin{pmatrix}
            \bP u\\ \eta_1\\ \eta_2\\ d_1\\ d_2
        \end{pmatrix} 
        &= A_\mfs \begin{pmatrix}
            \bP u\\ \eta_1\\ \eta_2\\ d_1\\ d_2
        \end{pmatrix} + \begin{pmatrix}
            \bP f\\ 0\\ \overline{g}\\ 0\\h
        \end{pmatrix}, \tfor t \in (0,T), \enspace 
        \begin{pmatrix}
            \bP u\\ \eta_1\\ \eta_2\\ d_1\\ d_2
        \end{pmatrix}(0)
        = \begin{pmatrix}
            \bP u_0\\ \eta_{1,0}\\ \eta_{2,0}\\ d_{1,0}\\ d_{2,0}
        \end{pmatrix}.
    \end{aligned}
\end{equation}

We now establish the maximal $\rL^p$-regularity, or, equivalently, the $\cR$-sectoriality of $A_\mfs$.

\begin{prop}\label{prop:cR-sect and max reg up to shift}
Consider $q \in (1,\infty)$.
Then there exists $\mu \ge 0$ such that the (shifted) multilayered fluid-structure operator $-A_\mfs + \mu$, with $A_\mfs$ as introduced in \eqref{eq:multilayered fluid-structure op}, is $\cR$-sectorial on $\rX_0$ with $\cR$-angle $\phi_{-A_\mfs + \mu}^{\cR} < \nicefrac{\pi}{2}$.
In particular, the operator $-A_\mfs + \mu$ has maximal $\rL^p$-regularity on $\rX_0$.
\end{prop}

\begin{proof}
The second part of the assertion is an immediate consequence of the first one by the characterization of maximal $\rL^p$-regularity on UMD spaces, such as the present $\rX_0$, as provided in \autoref{lem:char max reg via R-sect}, so we focus on the first part.

First, let us observe that the operator matrix $A_\mfs$ has a {\em non-diagonal} domain, which renders the analysis more complicated.
In order to ease this, and to split the considerations for the coupling of the thin and the thick structural layer with the fluid, we will use a decoupling approach.
More precisely, we will invoke a suitable similarity transform $S$, based on the lifting for the thick structural layer as explained in \autoref{ssec:lifting args}, and exploit that $\cR$-sectoriality is preserved under similarity transforms, see \autoref{lem:sim trafos}(a).

Let us now make the above considerations precise.
The aim is to consider $\td_1 = d_1 - D_\rs((P_\rm \eta_1) \mre_3)$ as well as $\td_2 = d_2 - D_\rs((P_\rm \eta_2) \mre_3)$.
The similarity transform thus takes the shape
\begin{equation*}
    \begin{aligned}
        S 
        &= \begin{pmatrix}
        \Id & 0 & 0 & 0 & 0\\
        0 & \Id & 0 & 0 & 0\\
        0 & 0 & \Id & 0 & 0\\
        0 & -D_\rs(P_\rm(\cdot) \mre_3) & 0 & \Id & 0\\
        0 & 0 & -D_\rs(P_\rm(\cdot) \mre_3) & 0 & \Id
    \end{pmatrix}, \,
    S^{-1} = \begin{pmatrix}
        \Id & 0 & 0 & 0 & 0\\
        0 & \Id & 0 & 0 & 0\\
        0 & 0 & \Id & 0 & 0\\
        0 & D_\rs(P_\rm(\cdot) \mre_3) & 0 & \Id & 0\\
        0 & 0 & D_\rs(P_\rm(\cdot) \mre_3) & 0 & \Id
    \end{pmatrix}
    \end{aligned}
\end{equation*}
for the inverse $S^{-1}$.
In order to simplify the subsequent calculations, we will also write 
\begin{equation*}
    S = \begin{pmatrix}
        \Id_\fs & 0\\
        T & \Id_\rs
    \end{pmatrix} \tand S^{-1} = \begin{pmatrix}
        \Id_\fs & 0\\
        -T & \Id_\rs
    \end{pmatrix}, \twith T = \begin{pmatrix}
        0 & -D_\rs(P_\rm(\cdot) \mre_3) & 0\\
        0 & 0 & -D_\rs(P_\rm(\cdot) \mre_3)
    \end{pmatrix}.
\end{equation*}
In this context, we introduce the space
\begin{equation}\label{eq:alt ground space Y_0}
    \rY_0 \coloneqq \rL_\sigma^q(\Omega_\rf) \times \rX_0^\rp \times \rX_0^\rs = \rL_\sigma^q(\Omega_\rf) \times \left(\rW_0^{2,q}(\omega) \cap \rL_\rm^q(\omega)\right) \times \rL_\rm^q(\omega) \times \rW_0^{1,q}(\Omega_\rs)^3 \times \rL^q(\Omega_\rs)^3.
\end{equation}
It readily follows that $\rY_0 = S \rX_0$ as well as $\rX_0 = S^{-1} \rY_0$.
Moreover, we write the multilayered fluid-structure operator $A_\mfs$ as $A_\mfs = \begin{pmatrix} A_\fs & B\\ C & A_\rs \end{pmatrix}$, where
\begin{equation}\label{eq:fluid-structure operator}
    \begin{aligned}
        A_\fs 
        &= \begin{pmatrix}
        A_0 & 0 & -A_0 \bP D_\fl\\
        0 & 0 & \Id\\
        M_s^{-1} \gamma_\rm N \Delta(\cdot) \cdot \nu & -M_s^{-1} P_\rm \Delta_\rp^2 & M_\rs^{-1} \Delta_\rp
        \end{pmatrix}, \enspace B = \begin{pmatrix}
            0 & 0\\
            0 & 0\\
            M_s^{-1}\cK & \delta M_s^{-1}\cK
        \end{pmatrix} \tand\\
        C 
        &= \begin{pmatrix}
            0 & 0 & 0\\
            0 & -L_0 D_\rs (P_\rm (\cdot) \mre_3) & -\delta L_0 D_\rs (P_\rm (\cdot) \mre_3)
        \end{pmatrix}.
    \end{aligned}
\end{equation}
At this stage, let us recall from \cite[Thm.~4.6]{MT:21} that the operator $A_\fs \colon \rD(A_\fs) \subset \rX_0^\fs \to \rX_0^\fs$ with domain
\begin{equation}\label{eq:domain fluid-structure operator}
    \rD(A_\fs) = \bigl\{(\bP u,\eta_1,\eta_2) \in \rW^{2,q}(\Omega_\rf)^3 \cap \rL_\sigma^q(\Omega_\rf) \times \rD(A_\rp) : \bP u - \bP D_\fl \eta_2 \in \rD(A_0)\bigr\}
\end{equation}
has the property that $-A_\fs$ is $\cR$-sectorial on $\rX_0^\fs = \rL_\sigma^q(\Omega_\rf) \times \rW_0^{2,q}(\omega) \cap \rL_\rm^q(\omega) \times \rL_\rm^q(\omega)$ with $\phi_{-A_\fs}^\cR < \nicefrac{\pi}{2}$.

Using the block operator structure of $A_\mfs$, we find that the operator $\tA_\mfs$ on $\rY_0$ takes the shape
\begin{equation}\label{eq:sim trafo multilayered fs op}
    \begin{aligned}
        \tA_\mfs 
        &= S A_\mfs S^{-1} = \begin{pmatrix}
            A_\fs - B T & B\\
            T A_\fs - T B T + C - A_\rs T & T B  + A_\rs
        \end{pmatrix}, \twith\\ \rD(\tA_\mfs) 
        &= \rD(A_\fs) \times \rD(A_\rs) \subset \rY_0,
    \end{aligned}
\end{equation}
where $\rD(A_\fs)$ and $\rD(A_\rs)$ have been made precise in \eqref{eq:domain fluid-structure operator} and \eqref{eq:op matrix thick structure}, respectively.
In particular, $\tA_\mfs$ has a {\em diagonal} domain when viewed as a block operator matrix, but it is of a more complicated shape.
In the following, we elaborate more on the precise shape of $\tA_\mfs$.
For simplicity of notation, we will only write~$D_\rs$ instead of $D_\rs(P_\rm(\cdot) \mre_3)$ in the sequel.
Now we calculate
\begin{equation*}
    \begin{aligned}
        B T
        &= \begin{pmatrix}
        0 & 0 & 0\\
        0 & 0 & 0\\
        0 & -M_s^{-1}\cK D_\rs & -\delta M_s^{-1}\cK D_\rs
        \end{pmatrix}, \enspace T B = \begin{pmatrix}
            0 & 0\\
            -D_\rs M_s^{-1} \cK & -\delta D_\rs M_s^{-1} \cK
        \end{pmatrix}.\\ 
        T A_\fs 
        &= \begin{pmatrix}
        0 & 0 & -D_\rs\\
        - D_\rs M_s^{-1} \gamma_\rm N \Delta(\cdot) \cdot \nu & D_\rs M_s^{-1} P_\rm \Delta_\rp^2 & -D_\rs M_\rs^{-1} \Delta_\rp
        \end{pmatrix},\\
        T B T
        &= \begin{pmatrix}
            0 & 0 & 0\\
            0 & D_\rs M_s^{-1} \cK D_\rs & \delta D_\rs M_s^{-1} \cK D_\rs
        \end{pmatrix} \tand
        A_\rs T
        = \begin{pmatrix}
            0 & 0 & -D_\rs\\
            0 & -L_0 D_\rs & -\delta L_0 D_\rs
        \end{pmatrix}.
    \end{aligned}
\end{equation*}
With these computations, we infer that $\tA_\mfs$ takes the shape
\begin{equation}\label{eq:transformed multilayered fs op}
    \tA_\mfs = \begin{pmatrix}
        A_\fs - B T & B\\
        \tC & T B + A_\rs
    \end{pmatrix}, \twhere 
\end{equation}
\begin{equation*}
    \begin{aligned}
        A_\fs - B T
        &= \begin{pmatrix}
            A_0 & 0 & -A_0 \bP D_\fl\\
            0 & 0 & \Id\\
            M_s^{-1} \gamma_\rm N \Delta(\cdot) \cdot \nu & -M_s^{-1} (P_\rm \Delta_\rp^2 - \cK D_\rs) & M_\rs^{-1}(\Delta_\rp + \delta \cK D_\rs)
        \end{pmatrix},\\
        \tC
        &= \begin{pmatrix}
            0 & 0 & 0\\
            -D_\rs M_s^{-1} \gamma_\rm N \Delta(\cdot) \cdot \nu & D_\rs M_s^{-1} (P_\rm \Delta_\rp^2 - \cK D_\rs) & - D_\rs M_s^{-1} (\Delta_\rp + \delta \cK D_\rs)
        \end{pmatrix} \tand\\
        T B + A_\rs
        &= \begin{pmatrix}
            0 & \Id\\
            L_0 - D_\rs M_s^{-1} \cK & \delta (L_0 - D_\rs M_s^{-1} \cK)
        \end{pmatrix}.
    \end{aligned}
\end{equation*}
First, it readily follows from the result on $A_\fs$ that we recalled above together with \autoref{prop:bdd Hinfty & spectral bound thick structural layer}(c) that the diagonal operator matrix $-\diag(A_\fs,A_\rs)$, with domain $\rD(\tA_\mfs) = \rD(A_\fs) \times \rD(A_\rs)$, is $\cR$-sectorial on $\rY_0$, with $\cR$-angle strictly less than $\nicefrac{\pi}{2}$.
Next, we consider the triangular block operator matrix $\tA_\mfs^{\triangle} \colon \rD(\tA_\mfs^{\triangle}) \subset \rY_0 \to \rY_0$ defined by
\begin{equation}\label{eq:triang matrix}
    \tA_\mfs^{\triangle} = \begin{pmatrix}
        A_\fs & 0\\
        \tC & A_\rs
    \end{pmatrix}, \twith \rD(\tA_\mfs^{\triangle}) = \rD(\tA_\mfs) = \rD(A_\fs) \times \rD(A_\rs).
\end{equation}
For $\lambda \in \rho(A_\fs)$, we then get the representation
\begin{equation}\label{eq:repr lambda - triang matrix}
    \lambda - \tA_\mfs^{\triangle} = \begin{pmatrix}
        \lambda - A_\fs & 0\\
        -\tC & \lambda - A_\rs
    \end{pmatrix} = \begin{pmatrix}
        \Id & 0\\
        -\tC R(\lambda,A_\fs) & \Id
    \end{pmatrix} \cdot \begin{pmatrix}
        \lambda - A_\fs & 0\\
        0 & \lambda - A_\rs
    \end{pmatrix}.
\end{equation}
Thus, by the $\cR$-sectoriality of $-\diag(A_\fs,A_\rs)$, we find that $- \tA_\mfs^{\triangle}$ is $\cR$-sectorial on $\rY_0$ with $\phi_{- \tA_\mfs^{\triangle}}^{\cR} < \nicefrac{\pi}{2}$.

With regard to \eqref{eq:transformed multilayered fs op}, it thus remains to discuss the part $\tB \coloneqq \begin{pmatrix}
        - B T & B\\
        0 & T B
    \end{pmatrix}$,
and to show that $\tB$ is a relatively bounded perturbation of $\tA_\mfs^{\triangle}$, with arbitrarily small bound.
In fact, as a preparation for the upcoming spectral theory, we will even verify that it is a relatively compact perturbation.
For this purpose, let us first observe that the embedding $\rD(L_0) = \rW^{2,q}(\Omega_\rs)^3 \cap \rW_0^{1,q}(\Omega_\rs)^3 \hookrightarrow \rW^{2-\eps,q}(\Omega)^3 \cap \rW_0^{1,q}(\Omega_\rs)^3$ is compact for all $\eps > 0$.
In the following, we consider $\eps > 0$ sufficiently small so that $1- \nicefrac{1}{q} - \eps > 0$.
By construction, it follows that $\cK \colon \rW^{2-\eps,q}(\Omega_\rs)^3 \cap \rW_0^{1,q}(\Omega_\rs)^3 \to \rW^{1-\nicefrac{1}{q}-\eps,q}(\omega)^3 \cap \rL_\rm^q(\omega)$ is bounded.
By \autoref{lem:props of added mass op}(a), this yields the relative compactness of $B$.
From \autoref{lem:props of added mass op}(a), we also recall that $M_s$ is an automorphism on $\rW^{1-\nicefrac{1}{q}-\eps,q}(\omega) \cap \rL_\rm^q(\omega)$.
On the other hand, we invoke the compactness of the embedding $\rW^{4,q}(\omega) \cap \rW_0^{2,q}(\omega) \cap \rL_\rm^q(\omega) \hookrightarrow  \rW_0^{2,q}(\omega) \cap \rL_\rm^q(\omega)$ joint with the boundedness of the operators 
\begin{equation*}
    \begin{aligned}
        D_\rs 
        &\colon \rW_0^{2,q}(\omega)^3 \cap \rL_\rm^q(\omega)^3 \to \rW^{2,q}(\Omega_\rs)^3, \enspace \cK \colon \rW^{2,q}(\Omega_\rs)^3 \to \rW^{1-\nicefrac{1}{q}-\eps,q}(\omega) \cap \rL_\rm^q(\omega) \tand\\
        M_s^{-1} 
        &\colon \rW^{1-\nicefrac{1}{q}-\eps,q}(\omega) \cap \rL_\rm^q(\omega) \to \rW^{1-\nicefrac{1}{q}-\eps,q}(\omega) \cap \rL_\rm^q(\omega).
    \end{aligned}
\end{equation*}
This implies the relative compactness of $B T$.
With regard to the relative compactness of $T B$, we additionally invoke that (the extension of) $D_\rs$ from \eqref{eq:lifting op Lame} maps $\rL^q(\omega)$ to $\rW^{s,q}(\Omega_\rs)^3$ for all $s < \nicefrac{1}{q}$.

In total, we have verified that $\tB$ is a relatively compact perturbation of $\tA_\mfs^{\triangle}$.
This already yields that~$\tB$ is a relatively bounded perturbation of $\tA_\mfs^{\triangle}$ with $\tA_\mfs^{\triangle}$-bound zero, see \cite[Lemma~III.2.16]{EN:00}.
Classical perturbation theory for $\cR$-sectorial operators, see, e.g., \cite[Cor.~2]{KW:01}, then yields that there exists $\mu \ge 0$ such that $-\tA_\mfs + \mu$ is $\cR$-sectorial on $\rY_0$, with $\phi_{-\tA_\mfs + \mu}^{\cR} < \nicefrac{\pi}{2}$.
By the \autoref{lem:sim trafos}(a) on the preservation of the $\cR$-sectoriality under similarity transforms, this completes the proof.
\end{proof}

\subsection{Spectral theory}\label{ssec:spectral theory}
\ 

In view of \autoref{prop:Arendt-Bu thm}, it remains to verify the invertibility of the multilayered fluid-structure operator $A_\mfs$ from \eqref{eq:multilayered fluid-structure op}.
Here we show the invertibility.
Due to the parabolic-hyperbolic damped character of the present problem, the spectral analysis is significantly more involved than for the problem without thick structure.
In fact, the operator $A_\mfs$ from \eqref{eq:multilayered fluid-structure op} {\em does not} have compact resolvent, so its spectrum does not necessarily only consist of eigenvalues, and we need to handle the essential spectrum.

In order to prepare for the following considerations, we collect and discuss several preliminary results in the sequel.
Let us start with spectral theory for the fluid-structure operator $A_\fs$ as recalled in \eqref{eq:fluid-structure operator} and \eqref{eq:domain fluid-structure operator}.
The following result on the spectrum of $A_\fs$ can be found in the proof of \cite[Thm.~4.6]{MT:21}.

\begin{lem}\label{lem:spectrum fluid-structure op}
Let $q \in (1,\infty)$.
Then $\sigma(A_\fs) \subset \{\lambda \in \C : \Rep \lambda < 0\}$, so $s(A_\fs) < 0$.
\end{lem}

Next, we elaborate on the point spectrum $\sigma_\rp$ of the multilayered fluid-structure operator $A_\mfs$.

\begin{lem}\label{lem:point spectrum multilayered fluid-structure op}
Consider $q \in (1,\infty)$, and recall $A_\mfs$ from \eqref{eq:multilayered fluid-structure op}.
Then $\sigma_\rp(A_\mfs) \subset \{\lambda \in \C : \Rep \lambda < 0\}$.
\end{lem}

\begin{proof}
First, let us observe that the point spectrum is independent of $q$, so we may consider $q = 2$ and thus test the eigenvalue problem $(\lambda - A_\mfs) (\bP u,\eta_1,\eta_2,d_1,d_2)^\top = 0$.
Similarly as in \autoref{prop:equiv reform lin probl}, we find that the eigenvalue problem can be equivalently rewritten as
\begin{equation}\label{eq:resolvent problem assoc with multilayered fs op}
    \left\{
    \begin{aligned}
        \lambda u - \mdiv \sigma_\rf(u,\pi)
        &= 0, &&\tin \Omega_\rf,\\
        \mdiv u
        &= 0, &&\tin \Omega_\rf,\\
        \lambda \eta_1 - \eta_2
        &= 0, &&\tin \omega,\\
        \lambda \eta_2 + P_\rm \Delta_s^2 \eta_1 - \Delta_s \eta_2
        &= \gamma_\rm \pi + P_\rm \bigl(\left.(\sigma_\rs(d_1,d_2) \mre_3)\right|_{\Gamma_0} \cdot \mre_3\bigr), &&\tin \omega,\\
        \int_\omega \eta_1 \srd s
        &= 0,\\
        \lambda d_1 - d_2 
        &= 0, &&\tin \Omega_\rs,\\
        \lambda d_2 - \mdiv \sigma_\rs(d_1,d_2)
        &= 0, &&\tin \Omega_\rs,\\
        u
        &= P_\rm(\eta_2) \mre_3, &&\ton \Gamma_0,\\
        d_1
        &= P_\rm(\eta_1) \mre_3, &&\ton \Gamma_0,\\
        u
        &= 0, &&\ton \Gamma_\rf,\\
        d_1
        &= 0, &&\ton \Gamma_\rs,\\
        \eta_1 = \nabla_s \eta_1 \cdot n_\omega
        &= 0, &&\ton \del \omega.
    \end{aligned}
    \right.
\end{equation}
At this stage, let us recall the modified trace operator $\gamma_\rm$ from \eqref{eq:mod trace op}, and we also invoke the cancellation of the other parts of the Cauchy stress for the fluid part as revealed in \eqref{eq:more comp repr kinematic cc in plate eq}.
Now, in order to determine the point spectrum, we test \eqref{eq:resolvent problem assoc with multilayered fs op}$_1$ by $\overline{u}$, \eqref{eq:resolvent problem assoc with multilayered fs op}$_4$ by $\overline{\eta}_2$ and \eqref{eq:resolvent problem assoc with multilayered fs op}$_7$ by $\overline{d}_2$ and integrate over the respective domains.
In fact, for \eqref{eq:resolvent problem assoc with multilayered fs op}$_1$, also invoking $\mdiv u = 0$ and the boundary conditions of $u$, we find that
\begin{equation}\label{eq:fluid eq tested}
    0 = \lambda \int_{\Omega_\rf} |u|^2 \srd x + 2 \int_{\Omega_\rf} |\D(u)|^2 \srd x + \int_{\Gamma_0} \pi \cdot \overline{P_\rm(\eta_2)} \srd s.
\end{equation}
Next, with regard to \eqref{eq:resolvent problem assoc with multilayered fs op}$_4$, we integrate by parts and additionally invoke that $\overline{\eta}_2 = \overline{\lambda} \overline{\eta}_1$ to deduce that
\begin{equation}\label{eq:plate eq tested}
    \lambda \int_\omega |\eta_2|^2 \srd s + \overline{\lambda} \int_\omega |\Delta_s \eta_1|^2 \srd s + \int_\omega |\nabla_s \eta_2|^2 \srd s = \int_\omega \gamma_\rm \pi \overline{\eta}_2 \srd s + \int_\omega P_\rm\left.\bigl(\sigma_\rs(d_1,d_2) \mre \bigr)\right|_{\Gamma_0} \cdot \mre_3 \overline{\eta}_2 \srd s.
\end{equation}
In order to combine \eqref{eq:fluid eq tested} and \eqref{eq:plate eq tested}, we observe that for functions $f$ and $g$ defined on $\omega$, and for $\overline{f}$ and~$\overline{g}$ representing the averages over $\omega$, we have $\int_\omega f P_\rm g \srd s = \int_\omega f g \srd s - |\omega| \overline{f} \overline{g} = \int_\omega (P_\rm f)  g \srd s$.
Thus, a concatenation of \eqref{eq:fluid eq tested} and \eqref{eq:plate eq tested} leads to
\begin{equation}\label{eq:tested fluid and plate eqs}
    \begin{aligned}
        &\lambda \int_{\Omega_\rf} |u|^2 \srd x + 2 \int_{\Omega_\rf} |\D(u)|^2 \srd x + \lambda \int_\omega |\eta_2|^2 \srd s + \overline{\lambda} \int_\omega |\Delta_s \eta_1|^2 \srd s\\
        &\quad + \int_\omega |\nabla_s \eta_2|^2 \srd s - \int_\omega P_\rm\left.\bigl(\sigma_\rs(d_1,d_2) \mre_3 \bigr)\right|_{\Gamma_0} \cdot \mre_3 \overline{\eta}_2 \srd s = 0.
    \end{aligned}
\end{equation}

It remains to handle \eqref{eq:resolvent problem assoc with multilayered fs op}$_7$.
First, we infer that $0 = \lambda \int_{\Omega_\rs} |d_2|^2 \srd x - \int_{\Omega_\rs} \mdiv \sigma_\rs (d_1,d_2) \cdot \overline{d}_2 \srd x$.
We recall~$\sigma_\rs$ from \eqref{eq:Piola-Kirchhoff stress tensor}, where we further introduce the notation $S_\rs(f) \coloneqq \mu_\rs\bigl((\nabla f + (\nabla f)^\top)\bigr) + \lambda_\rs \mdiv(f) \Id_3$, i.e., $\sigma_\rs(d_1,d_2) = S_\rs(d_1 + \delta d_2)$, and use $\lambda d_1 = d_2$ in \eqref{eq:resolvent problem assoc with multilayered fs op}$_6$ for
\begin{equation*}
    \begin{aligned}
        - \int_{\Omega_\rs} \mdiv \sigma_\rs (d_1,d_2) \overline{d}_2 \srd x
        &= -\int_{\Omega_\rs} \cL d_1 \cdot \overline{d}_2 \srd x - \delta \int_{\Omega_\rs} \cL d_2 \cdot \overline{d}_2 \srd x\\
        &= \overline{\lambda} \biggl[2 \mu_\rs \int_{\Omega_\rs} |\D(d_1)|^2 \srd x + \lambda_\rs \int_{\Omega_\rs} |\mdiv d_1|^2 \srd x - \int_{\del \Omega_\rs} S_\rs(d_1) \overline{d}_1 \cdot n\srd s\biggr]\\
        &\quad + \delta \biggl[2 \mu_\rs \int_{\Omega_\rs} |\D(d_2)|^2 \srd x + \lambda_\rs \int_{\Omega_\rs} |\mdiv d_2|^2 \srd x - \int_{\del \Omega_\rs} S_\rs(d_2) \overline{d}_2 \cdot n\srd s\biggr],
    \end{aligned}
\end{equation*}
where $n$ denotes the outer unit normal vector to $\del \Omega_\rs$.
In the following, we elaborate on the boundary terms.
Making use of $\overline{\lambda} \cdot \overline{d}_1 = \overline{d}_2$ as well as the boundary conditions for $d_1$, including $d_1 = P_\rm(\eta_1) \mre_3$ on the interface, and exploiting $\overline{\lambda} \overline{\eta}_1 = \overline{\eta}_2$, which results from \eqref{eq:resolvent problem assoc with multilayered fs op}$_3$, we deduce that
\begin{equation*}
    \begin{aligned}
        \\&\quad -\overline{\lambda} \int_{\del \Omega_\rs} S_\rs(d_1) \overline{d}_1 \cdot n \srd x - \delta \int_{\del \Omega_\rs} S_\rs(d_2) \overline{d}_2 \cdot n\srd x 
        = -\int_{\Gamma_0} S_\rs(d_1) \overline{d}_2 \cdot n \srd s - \delta \int_{\Gamma_0} S_\rs(d_2) \overline{d}_2 \cdot n \srd s\\
        &= - \overline{\lambda} \int_{\Gamma_0} S_\rs(d_1 + \delta d_2) \overline{d}_1 \cdot n \srd s
        = -\overline{\lambda} \int_{\Gamma_0} \sigma_\rs(d_1,d_2) \overline{P_\rm(\eta_1) \mre_3} \srd s
        = -\int_{\Gamma_0} \sigma_\rs(d_1,d_2) P_\rm(\overline{\eta}_2) \mre_3 \srd s.
    \end{aligned}
\end{equation*}
Putting together the latter calculations with \eqref{eq:tested fluid and plate eqs}, and employing similar arguments as above for the cancellation of the boundary terms, we conclude that
\begin{equation*}
    \begin{aligned}
        0
        &= \lambda \int_{\Omega_\rf} |u|^2 \srd x + 2 \int_{\Omega_\rf} |\D(u)|^2 \srd x + \lambda \int_\omega |\eta_2|^2 \srd s + \overline{\lambda} \int_\omega |\Delta_s \eta_1|^2 \srd s + \int_\omega |\nabla_s \eta_2|^2 \srd s + \lambda \int_{\Omega_\rs} |d_2|^2 \srd x\\
        &\quad + \overline{\lambda} \biggl[2 \mu_\rs \int_{\Omega_\rs} |\D(d_1)|^2 \srd x + \lambda_\rs \int_{\Omega_\rs} |\mdiv d_1|^2 \srd x\biggr] + \delta \biggl[2 \mu_\rs \int_{\Omega_\rs} |\D(d_2)|^2 \srd x + \lambda_\rs \int_{\Omega_\rs} |\mdiv d_2|^2 \srd x\biggr].
    \end{aligned}
\end{equation*}
The latter relation reveals that $\Rep \lambda \le 0$.
Moreover, if $\lambda = 0$, then $\D(u) = 0$, $\nabla_s \eta_2 = 0$, $\D(d_2) = 0$ as well as $\mdiv d_2 = 0$.
At first, this yields that $u$, $\eta_2$ and $d_2$ are constant, so the boundary conditions imply that $u = d_ 2 = 0$ and $\eta_2 = 0$, which in turn results in $d_1 = 0$ as well as $\eta_1 = 0$ thanks to the boundary conditions of $d_1$ and $\eta_1$.
This shows that $0 \notin \sigma_\rp(A_\mfs)$ and thus completes the proof of the lemma in conjunction with the previous considerations.
\end{proof}

The following result asserting that the spectral bound of the multilayered fluid-structure operator $A_\mfs$ is negative is of central importance.

\begin{prop}\label{prop:spectral theory multilayered fs op}
Let $q \in (1,\infty)$, and recall the multilayered fluid-structure operator $A_\mfs$ from \eqref{eq:multilayered fluid-structure op}.
Then there exists $\omega < 0$ such that $\sigma(A_\mfs) \subset \{\lambda \in \C : \Rep \lambda \le \omega\}$, so the spectral bound $s(A_\mfs)$ especially satisfies $s(A_\mfs) < 0$.
\end{prop}

\begin{proof}
The main idea is to perform the spectral analysis for the operator $\tA_\mfs$ from \eqref{eq:sim trafo multilayered fs op} on the ground space $\rY_0$ as introduced in \eqref{eq:alt ground space Y_0}.
Note that $\tA_\mfs$ has the same spectrum as $A_\mfs$, because it results from $A_\mfs$ by a similarity transform, i.e., $\sigma(\tA_\mfs) = \sigma(A_\mfs)$.
The same is valid for the point spectrum, so we also have $\sigma_\rp(\tA_\mfs) = \sigma_\rp(A_\mfs)$.
We shall then successively derive a good control of the essential spectrum $\sigma_\ess(\tA_\mfs)$ of $\tA_\mfs$ and combine with the latter observation as well as \autoref{lem:point spectrum multilayered fluid-structure op} to reveal that $\sigma(A_\mfs) = \sigma(\tA_\mfs) = \sigma_\ess(\tA_\mfs) \cup \sigma_{\rd}(\tA_\mfs)$, where $\sigma_{\rd}$ represents the discrete spectrum and has the property that $\sigma_\rd \subset \sigma_\rp$.

First, from \eqref{eq:transformed multilayered fs op}, we recall the transformed multilayered fluid-structure operator 
\begin{equation*}
    \tA_\mfs = \begin{pmatrix}
        A_\fs - B T & B\\
        \tC & T B + A_\rs
    \end{pmatrix}.
\end{equation*}
We now successively study the spectrum of the operator $\tA_\mfs$.
To this end, let us begin with the diagonal operator matrix $\diag(A_\fs,A_\rs)$ with domain $\rD(\tA_\mfs) = \rD(A_\fs) \times \rD(A_\rs)$.
From \autoref{lem:spectrum fluid-structure op} together with \autoref{prop:bdd Hinfty & spectral bound thick structural layer}, it first follows that $\sigma(\diag(A_\fs,A_\rs)) \subset \{\lambda \in \C : \Rep \lambda < 0\}$, so~$s(\diag(A_\fs,A_\rs)) < 0$.
Next, we recall from \eqref{eq:triang matrix} and \eqref{eq:repr lambda - triang matrix} the operator matrix $\tA_\mfs^{\triangle} \colon \rD(\tA_\mfs^{\triangle}) \subset \rY_0 \to \rY_0$ and the representation
\begin{equation*}
    \lambda - \tA_\mfs^{\triangle} = \begin{pmatrix}
        \lambda - A_\fs & 0\\
        -\tC & \lambda - A_\rs
    \end{pmatrix} = \begin{pmatrix}
        \Id & 0\\
        -\tC R(\lambda,A_\fs) & \Id
    \end{pmatrix} \cdot \begin{pmatrix}
        \lambda - A_\fs & 0\\
        0 & \lambda - A_\rs
    \end{pmatrix}, \tfor \lambda \in \rho(A_\fs).
\end{equation*}
For $\lambda \in \rho(\diag(A_\fs,A_\rs))$, we find $\lambda \in \rho(\tA_\mfs^{\triangle})$, so $\sigma(\tA_\mfs^{\triangle}) \subset \{\lambda \in \C : \Rep \lambda < 0\}$ and $s(\tA_\mfs^{\triangle}) < 0$.

With regard to \eqref{eq:transformed multilayered fs op}, it thus remains to discuss $\tB = \begin{pmatrix}
        - B T & B\\
        0 & T B
    \end{pmatrix}$.
In the proof of \autoref{prop:cR-sect and max reg up to shift}, we have already verified that $\tB$ is a relatively compact perturbation of $\tA_\mfs^{\triangle}$.
From \cite[Thm.~5.35]{Kat:95} and the above considerations, we infer that
\begin{equation*}
    \sigma_\ess(\tA_\mfs) = \sigma_\ess(\tA_\mfs^{\triangle}) \subset \sigma(\tA_\mfs^{\triangle}) \subset \{\lambda \in \C : \Rep \lambda < 0\}.
\end{equation*}
On the other hand, thanks to \autoref{lem:point spectrum multilayered fluid-structure op}, we have
\begin{equation*}
    \sigma_\rd(\tA_\mfs) \subset \sigma_\rp(\tA_\mfs) = \sigma_\rp(A_\mfs) \subset \{\lambda \in \C : \Rep \lambda < 0\}.
\end{equation*}
Putting together the previous two inclusions, we find that
\begin{equation*}
    \sigma(A_\mfs) = \sigma(\tA_\mfs) = \sigma_\ess(\tA_\mfs) \cup \sigma_\rd(\tA_\mfs) \subset \{\lambda \in \C : \Rep \lambda < 0\}.
\end{equation*}
This also implies that $s(A_\mfs) < 0$, showing the assertion of the proposition.
\end{proof}

The above proposition also allows us to deduce that $-A_\mfs$ admits maximal $\rL^p$-regularity on $\rX_0$ without shift, and that the associated semigroup is exponentially stable.

\begin{cor}\label{cor:max reg & exp stab of multilayered fs op}
Let $q \in (1,\infty)$, and recall the multilayered fluid-structure operator $A_\mfs$ from \eqref{eq:multilayered fluid-structure op}.
Then $-A_\mfs$ has maximal $\rL^p$-regularity, and the associated semigroup $\mre^{t A_\mfs}$ is exponentially stable, i.e., there exist $\mu_0 > 0$ and $C > 0$ such that $\| \mre^{t A_\mfs} w \|_{\rX_0} \le C \mre^{-\mu_0 t} \cdot \| w \|_{\rX_0}$ for all $w \in \rX_0$.
\end{cor}

\subsection{Maximal periodic regularity of the multilayered fluid-structure operator}
\ 

In this subsection, we concatenate the findings from the previous subsections in order to deduce the maximal periodic $\rL^p$-regularity of the multilayered fluid-structure operator $A_\mfs$.
Indeed, with regard to the Arendt-Bu theorem as asserted in \autoref{prop:Arendt-Bu thm}, the result below follows from the maximal $\rL^p$-regularity as revealed in \autoref{prop:cR-sect and max reg up to shift} together with the spectral analysis from \autoref{prop:spectral theory multilayered fs op}.

\begin{thm}\label{thm:max per reg multilayered fs op}
Let $q \in (1,\infty)$, and consider the operator $A_\mfs$ as introduced in \eqref{eq:multilayered fluid-structure op}.
Then $-A_\mfs$ has maximal periodic $\rL^p$-regularity on $\rX_0$.
\end{thm}

With regard to the upcoming fixed point argument, we slightly rephrase \autoref{thm:max per reg multilayered fs op} in the sequel so that we obtain an assertion on the linearized problem \eqref{eq:lin per multilayered interaction probl}.
This is done in the corollary below.

\begin{cor}\label{cor:max per reg result multilayered}
Let $f \in \rL^p(0,T;\rL^q(\Omega_\rf)^3) \eqqcolon \E_0^u$, $g \in \rL^p(0,T;\rL^q(\omega)) \eqqcolon \E_0^{\eta}$, $h \in \rL^p(0,T;\rL^q(\Omega_\rs)^3) \eqqcolon \E_0^d$.
Then the linearized problem \eqref{eq:lin per multilayered interaction probl} has a unique solution $v = (u,\pi,\eta_1,\eta_2,d_1,d_2) \in \E_1$, where $\E_1$ was introduced in \eqref{eq:max reg space}, and there exists a constant $C$, which only depends on $p$, $q$ and the geometry of the reference domain, with $\| v \|_{\E_1} \le C\bigl(\| f \|_{\E_0^u} + \| g \|_{\E_0^\eta} + \| h \|_{\E_0^d}\bigr)$.
\end{cor}

\begin{proof}
First, let us show that $\overline{g}$ from \eqref{eq:reform in op form multilayered fs op} satisfies $\overline{g} \in \rL^p(0,T;\rL_\rm^q(\omega))$.
For this purpose, we invoke that the projection $P_\rm$ from \eqref{eq:proj P_m} maps onto $\rL_\rm^q(\omega)$, and by \autoref{lem:props of added mass op}, $M_s$ is an automorphism on~$\rL_\rm^q(\omega)$, so $M_s^{-1} P_\rm g \in \rL^p(0,T;\rL_\rm^q(\omega))$.
Concerning the second addend of $\overline{g}$, we observe that the Neumann operator $N_2$ made precise in \eqref{eq:lifting op N_2} satisfies $N_2 \in \cL(\rL^q(\Omega_\rf)^3,\rW^{1,q}(\Omega_\rf) \cap \rL_\rm^q(\Omega_\rf))$, while the modified trace $\gamma_\rm$ from \eqref{eq:mod trace op} guarantees that $\gamma_\rm N_2 f \in \rL^p(0,T;\rL_\rm^q(\omega))$.
Thus, \autoref{lem:props of added mass op} implies that~$\overline{g} \in \rL^p(0,T;\rL_\rm^q(\omega))$.

By construction, it readily follows that $\bP f \in \rL^p(0,T;\rL_\sigma^q(\Omega_\rf))$.
Therefore, together with the preceding argument, we conclude that $(\bP f,0,\overline{g},0,h) \in \rL^p(0,T;\rX_0)$, where $\rX_0$ is as defined in \eqref{eq:ground space multilayered}.
For the domain~$\rD(A_\mfs)$ of the multilayered fluid-structure operator as introduced in \eqref{eq:dom multilayered fluid-structure op}, it then follows that there exists a unique solution $(\bP u,\eta_1,\eta_2,d_1,d_2) \in \rW^{1,p}(0,T;\rX_0) \cap \rL^p(0,T;\rD(A_\mfs))$ to the time-periodic analogue of \eqref{eq:reform in op form multilayered fs op}.
It remains to argue that $u$ and $\pi$ possess the desired regularities.
To this end, from \autoref{prop:equiv reform lin probl}, we recall $(\Id - \bP)u = (\Id - \bP)D_\fl \eta_2$, where $D_\fl$ has been introduced in \eqref{eq:lifting ops fluid & pressure}.
Moreover, from \autoref{lem:sol stat Stokes probl with inhom bc}, we recall that $D_\fl \in \cL(\rW_0^{2,q}(\omega) \cap \rL_\rm^q(\omega),\rW^{2,q}(\Omega_\rf)^3)$.
On the other hand, arguing as in \cite[(4.27)]{MT:21}, we get $D_\fl \in \cL( \rL_\rm^q(\omega),\rL^q(\Omega_\rf)^3)$.
In total, this reveals that $u \in \E_1^u$.

For the pressure, we invoke $\pi = N(\Delta \bP u \cdot \nu) - N_1 \eta_2' + N_2 f$ from \autoref{prop:equiv reform lin probl}.
From \eqref{eq:mapping props of Neumann op}$_1$ and~\eqref{eq:mapping props of Neumann op}$_2$, we get $N(\Delta \bP u \cdot \nu) \in \rL^p(0,T;\rW^{1,q}(\Omega_\rf) \cap \rL_\rm^q(\Omega_\rf))$ and $N_1 \eta_2' \in \rL^p(0,T;\rW^{1,q}(\Omega_\rf) \cap \rL_\rm^q(\Omega_\rf))$, respectively, while the mapping properties of $N_2$ from \eqref{eq:lifting op N_2} yield that $N_2 f \in \rL^p(0,T;\rW^{1,q}(\Omega_\rf) \cap \rL_\rm^q(\Omega_\rf))$.
In summary, we have verified that $\pi \in \E_1^\pi$.
\end{proof}

Thanks to \autoref{cor:max per reg result multilayered} on the solvability and the associated a priori estimate for the linearized time-periodic problem, it will be sufficient to estimate the nonlinear terms in relatively simple spaces.

\section{Existence of a strong time-periodic solution}\label{sec:ex of strong time per sol}

This section is dedicated to showing the main result on the existence of a time-periodic strong solution to the multilayered fluid-structure interaction problem \eqref{eq:multilayered fsi}.
The first part of this section deals with the nonlinear estimates, while in the second part, we prove \autoref{thm:ex of strong per sol} by means of a fixed point argument.

\subsection{Estimates of the nonlinear terms}\label{ssec:nonlin ests}
\ 

In this subsection, we discuss estimates of terms related to the transform in order to prepare for the nonlinear estimates.
For $R > 0$ and $\E_1$ as introduced in \eqref{eq:max reg space}, we denote by $\oB_{\E_1}(0,R)$ the closed ball in~$\E_1$ with center zero and radius $R > 0$.

We start by establishing estimates of the maximal regularity space in the lemma below.
The assertion of~(a) is a consequence of the mixed derivative theorem, see for example \cite[Thm.~4.5.10]{PS:16}, while the assertion of~(b) follows from \cite[Thm.~III.4.10.2]{Ama:95}.

\begin{lem}\label{lem:embs of the max reg space}
Let $p$, $q \in (1,\infty)$, and recall $\tE_1$ from \eqref{eq:max reg space}.
\begin{enumerate}[(a)]
    \item For all $\theta \in (0,1)$, it holds that
    \begin{equation*}
        \tE_1 \hookrightarrow \rH^{\theta,p}\bigl(0,T;\rH^{2(1-\theta),q}(\Omega_\rf)^3 \times \rH^{2 + 2(1-\theta),q}(\omega) \times \rH^{2(1-\theta),q}(\omega) \times \rH^{1,q}(\Omega_\rs)^3 \times \rH^{1-\theta,q}(\Omega_\rs)^3\bigr).
    \end{equation*}
    \item It holds that
    \begin{equation*}
        \tE_1 
        \hookrightarrow \BUC\bigl([0,T];\rB_{qp}^{2-\nicefrac{2}{p}}(\Omega_\rf)^3 \times \rB_{qp}^{4-\nicefrac{2}{p}}(\omega) \times \rB_{qp}^{2-\nicefrac{2}{p}}(\omega) \times \rW^{1,q}(\Omega_\rs)^3 \times \rB_{qp}^{1-\nicefrac{1}{p}}(\Omega_\rs)^3\bigr).
    \end{equation*}
\end{enumerate}
\end{lem}

In the lemma below, we collect useful estimates of the fluid velocity $u$ and the plate displacement $\eta$.

\begin{lem}\label{lem:ests of plate and fluid vars}
Let $p$, $q \in (1,\infty)$ be such that $\frac{1}{p} + \frac{3}{2q} < \frac{3}{2}$, and consider $v = (u,\pi,\eta_1,\eta_2,d_1,d_2) \in \oB_{\E_1}(0,R)$.
\begin{enumerate}[(a)]
    \item Then there exists $C > 0$ such that $\| \eta_1 \|_{\rL^\infty(0,T;\rC^1(\overline{\omega}))} \le C R$.
    \item There is $C > 0$ with $\| u \|_{\rL^{3p}(0,T;\rL^{3q}(\Omega_\rf))} \le C R$. 
    \item For a constant $C > 0$, it holds that $\| \nabla_s^2 \eta_1 \|_{\rL^{3p}(0,T;\rL^{3q}(\omega))} + \| \eta_2 \|_{\rL^{3p}(0,T;\rL^{3q}(\omega))} \le C R$.
    \item There exists $C > 0$ such that
    \begin{equation*}
        \| \nabla u \|_{\rL^{\nicefrac{3p}{2}}(0,T;\rL^{\nicefrac{3q}{2}}(\Omega_\rf))} + \| \nabla_s^3 \eta_1 \|_{\rL^{\nicefrac{3p}{2}}(0,T;\rL^{\nicefrac{3q}{2}}(\omega))} + \| \nabla_s \eta_2 \|_{\rL^{\nicefrac{3p}{2}}(0,T;\rL^{\nicefrac{3q}{2}}(\omega))} \le C R.
    \end{equation*}
    \item There is $C > 0$ with $\| \nabla u \|_{\rL^p(0,T;\rL^q(\del \Omega_\rf))} + \| u \|_{\rL^{\nicefrac{3p}{2}}(0,T;\rL^{\nicefrac{3q}{2}}(\del \Omega_\rf))} \le C R$.
\end{enumerate}
\end{lem}

Before proving the preceding lemma, we comment on an important consequence of the assertion in~(a).

\begin{rem}\label{rem:smallness of eta and diffeo}
By choosing $R > 0$ sufficiently small, we can ensure that $\eta = \eta_1$ satisfies the smallness condition \eqref{eq:smallness ass displacement} for $\delta_0$ made precise in \eqref{eq:concr bound for plate displacement}.
By the discussion in \autoref{sec:trafo to fixed dom}, this implies that~$X(t,\cdot)$ is a~$\rC^1$-diffeomorphism from $\Omega_\rf$ onto $\Omega_\rf(t)$ for all $t \in [0,T]$.
\end{rem}

\begin{proof}[Proof of \autoref{lem:ests of plate and fluid vars}]
The assertion of~(a) is a consequence of \autoref{lem:embs of the max reg space}(b) and $\rB_{qp}^{4-\nicefrac{2}{p}}(\omega) \hookrightarrow \rC^1(\overline{\omega})$, following from Sobolev embeddings, see for instance \cite[Thm.~4.6.1]{Tri:78}, thanks to the assumptions on $p$ and~$q$ that especially imply $\frac{1}{p} + \frac{1}{q} < \frac{3}{2}$, or, equivalently, $4 - \frac{2}{p} - \frac{2}{q} > 1$.

For~(b), we observe that by Sobolev embeddings, $\rH^{\theta,p}(0,T;\rH^{2(1-\theta),q}(\Omega_\rf)^3) \hookrightarrow \rL^{3p}(0,T;\rL^{3q}(\Omega_\rf)^3)$ requires that $\theta - \frac{1}{p} > -\frac{1}{3p}$ and $2(1-\theta) - \frac{3}{q} > -\frac{3}{3q}$.
Straightforward calculations reveal that such $\theta \in (0,1)$ can exist provided $\frac{1}{p} + \frac{3}{2q} < \frac{3}{2}$, so the assertion is implied by \autoref{lem:embs of the max reg space}(a).
Let us observe that the assertion of~(c) can be deduced in an analogous manner upon invoking \autoref{lem:embs of the max reg space}(a).

With regard to~(d), we first observe that $\| \nabla u \|_{\rL^{\nicefrac{3p}{2}}(0,T;\rL^{\nicefrac{3q}{2}}(\Omega_\rf))} \le \| u \|_{\rL^{\nicefrac{3p}{2}}(0,T;\rW^{1,\nicefrac{3q}{2}}(\Omega_\rf))}$.
Thus, we need to make sure that $\rH^{\theta,p}(0,T;\rH^{2(1-\theta),q}(\Omega_\rf)^3) \hookrightarrow \rL^{\nicefrac{3p}{2}}(0,T;\rW^{1,\nicefrac{3q}{2}}(\Omega_\rf)^3)$, i.e., by Sobolev embeddings, we need that $\theta - \frac{1}{p} > -\frac{2}{3p}$ and $2(1-\theta) - \frac{3}{q} > 1 - \frac{2}{q}$.
The existence of such $\theta \in (0,1)$ can then be guaranteed provided $\frac{1}{p} + \frac{3}{2q} < \frac{3}{2}$ is satisfied.
As a result, we conclude the estimate of the first term from \autoref{lem:embs of the max reg space}(a).
The estimates of the other two factors are again implied in the same manner.

For the first term in~(e), we make use of classical properties of the trace as well as the definition of the maximal regularity space $\E_1$ to argue that for $\eps > 0$ sufficiently small, we get
\begin{equation*}
    \| \nabla u \|_{\rL^p(0,T;\rL^q(\del \Omega_\rf))} \le C \cdot \| u \|_{\rL^p(0,T;\rW^{1+\nicefrac{1}{q}+\eps,q}(\Omega_\rf))} \le C \cdot \| u \|_{\rL^p(0,T;\rW^{2,q}(\Omega_\rf))} \le C \cdot \| v \|_{\E_1} \le C R.
\end{equation*}
The estimate of the second term is a consequence of the proof of the estimate of the first term in~(d) upon noting that the trace is in particular continuous from $\rW^{1,\nicefrac{3q}{2},q}(\Omega_\rf)^3$ to $\rL^{\nicefrac{3q}{2}}(\del\Omega_\rf)^3$.
\end{proof}

The last preparatory step before tackling the estimates of the nonlinear terms is to show suitable estimates of the transform.
These estimates can be derived in a similar way as in \cite[Prop.~5.2 and Prop.~6.2]{MT:21} upon employing \autoref{lem:ests of plate and fluid vars}(a) and the relation $a(X) = \frac{\nabla X}{\det(\nabla X)}$.

\begin{lem}\label{lem:est of terms related to the trafo}
Consider $p$, $q \in (1,\infty)$ with $\frac{1}{p} + \frac{3}{2q} < \frac{3}{2}$ as well as $v = (u,\pi,\eta_1,\eta_2,d_1,d_2) \in \oB_{\E_1}(0,R)$, and recall the diffeomorphism $X$, its inverse $Y$, $a=\Cof\left(\nabla Y\right)^\top$ and $b=\Cof\left(\nabla X\right)^\top$.
Then
\begin{enumerate}[(a)]
    \item there is $C > 0$ such that
    \begin{equation*}
        \| \nabla X - I_3 \|_{\rL^\infty(0,T;\rC(\overline{\Omega_\rf}))} + \| \nabla Y - I_3 \|_{\rL^\infty(0,T;\rC(\overline{\Omega_\rf}))} + \| a - I_3 \|_{\rL^\infty(0,T;\rC(\overline{\Omega_\rf}))} + \| b - I_3 \|_{\rL^\infty(0,T;\rC(\overline{\Omega_\rf}))} \le C R,
    \end{equation*}
    \item there exists $C > 0$ with $\| \det(\nabla X) - 1 \|_{\rL^\infty(0,T;\rC(\overline{\Omega_\rf}))} \le C R$ and $\frac{1}{2} \le \| \det \nabla X \|_{\rL^\infty(0,T;\rC(\overline{\Omega_\rf}))} \le \frac{3}{2}$,
    \item for some constant $C > 0$, it holds that
    \begin{equation*}
        \| \nabla X \|_{\rL^\infty(0,T;\rC(\overline{\Omega_\rf}))} + \| \nabla Y \|_{\rL^\infty(0,T;\rC(\overline{\Omega_\rf}))} + \| a \|_{\rL^\infty(0,T;\rC(\overline{\Omega_\rf}))} + \| b \|_{\rL^\infty(0,T;\rC(\overline{\Omega_\rf}))} \le C, \tand
    \end{equation*}
    \item there is $C > 0$ such that for all $i$, $j$, $k$, $l \in \{1,2,3\}$, we have
    \begin{equation*}
        \begin{aligned}
            \left|\frac{\del a_{ik}}{\del x_j}(X)\right| 
            &\le C \cdot |\nabla^2 X| \le C \bigl(|\eta_1| + |\nabla_s \eta_1| + |\nabla_s^2 \eta_1|\bigr),\\
            \left|\frac{\del^2 a_{ik}}{\del x_j^2}(X)\right| 
            &\le C\bigl(|\nabla X|^2 + |\nabla^3 X|\bigr) \le C\Bigl(\bigl(|\eta_1| + |\nabla_s \eta_1| + |\nabla_s^2 \eta_1|\bigr)^2 + |\nabla_s^3 \eta_1|\Bigr),\\
            |\del_t a(X)|
            &\le C\bigl(|\nabla^2 X| + |\nabla \del_t X|\bigr) \le C\bigl(|\eta_1| + |\nabla_s \eta_1| + |\nabla_s^2 \eta_1| + |\nabla_s \eta_2|\bigr),\\
            |\del_t Y(X)|
            &\le C \cdot |\del_t X| \le C \cdot |\eta_2|,\\
            \left|\frac{\del^2 Y_l}{\del x_j^2}(X)\right| 
            &\le C \cdot |\nabla^2 X| \le C\bigl(|\eta_1| + |\nabla_s \eta_1| + |\nabla_s^2 \eta_1|\bigr).
        \end{aligned}
    \end{equation*}
\end{enumerate}
\end{lem}

We are now in the position to estimate the nonlinear terms $F$ and $G$ from \autoref{sec:trafo to fixed dom}.

\begin{prop}\label{prop:self-map ests of F & G}
Let $p$, $q \in (1,\infty)$ with $\frac{1}{p} + \frac{3}{2q} < \frac{3}{2}$, and consider $v = (u,\pi,\eta_1,\eta_2,d_1,d_2) \in \oB_{\E_1}(0,R)$.
Then there exists $C > 0$ such that for the terms $F$ and $G$ as introduced in \eqref{eq:RHS F} and \eqref{eq:RHS G}, respectively, it holds that
\begin{equation*}
    \| F(u,\pi,\eta_1,\eta_2) \|_{\E_0^u} + \| G(u,\pi,\eta_1,\eta_2) \|_{\E_0^\eta} \le C R^2.
\end{equation*}
\end{prop}

\begin{proof}
Let us start with $F$.
We treat the addends separately.
For the first addend $b_{\alpha i}\frac{\partial^{2}a_{ik}}{\partial x_{j}^{2}}\left(X\right)u_{k}$, we use H\"older's inequality and invoke the boundedness of $b$ in $\rL^\infty(0,T;\rL^\infty(\Omega_\rf))$ by \autoref{lem:est of terms related to the trafo}(c).
Moreover, we employ the estimates of $\eta_1$, $\nabla_s \eta_1$ and $\nabla_s^2 \eta_1$ in $\rL^{3p}(0,T;\rL^{3q}(\omega))$ by $C R$ by \autoref{lem:ests of plate and fluid vars}(c) as well as the estimate of $\nabla_s^3 \eta_1$ in $\rL^{\nicefrac{3p}{2}}(0,T;\rL^{\nicefrac{3q}{2}}(\omega))$ by $C R$ by \autoref{lem:ests of plate and fluid vars}(d) to deduce from \autoref{lem:est of terms related to the trafo}(d) that
\begin{equation*}
    \left\| \frac{\del^2 a_{ik}}{\del x_j^2}(X) \right\|_{\rL^{\nicefrac{3p}{2}}(0,T;\rL^{\nicefrac{3q}{2}}(\omega))} \le C R.
\end{equation*}
Together with the estimate of $u$ in $\rL^{3p}(0,T;\rL^{3q}(\Omega_\rf))$ by $C R$ by \autoref{lem:ests of plate and fluid vars}(b), we conclude that
\begin{equation*}
    \left\| b_{\alpha i}\frac{\partial^{2}a_{ik}}{\partial x_{j}^{2}}\left(X\right)u_{k} \right\|_{\rL^p(0,T;\rL^q(\Omega_\rf))} \le C R^2.
\end{equation*}

In the next addend $b_{\alpha i}\frac{\partial a_{ik}}{\partial x_{j}}\left(X\right)\frac{\partial u_{k}}{\partial y_{l}}\frac{\partial Y_{l}}{\partial x_{j}}\left(X\right)$, we can bound the factor $b_{\alpha i}$ in the same way, and by \autoref{lem:est of terms related to the trafo}(c), we can also bound $\frac{\partial Y_{l}}{\partial x_{j}}$ in $\rL^\infty(0,T;\rL^\infty(\Omega_\rf))$.
On the other hand, thanks to \autoref{lem:est of terms related to the trafo}(d) combined with the estimates of $\eta_1$, $\nabla_s \eta_1$ and $\nabla_s^2 \eta_1$ in $\rL^{3p}(0,T;\rL^{3q}(\omega))$ by $C R$ by \autoref{lem:ests of plate and fluid vars}(c), we get
\begin{equation}\label{eq:est of der of a}
    \left\| \frac{\partial a_{ik}}{\partial x_{j}}\left(X\right) \right\|_{\rL^{3p}(0,T;\rL^{3q}(\Omega_\rf))} \le C R.
\end{equation}
Additionally using the estimate of $\nabla u$ in $\rL^{\nicefrac{3p}{2}}(0,T;\rL^{\nicefrac{3q}{2}}(\Omega_\rf))$ by $C R$ thanks to \autoref{lem:ests of plate and fluid vars}(d), and using H\"older's inequality, we find that
\begin{equation*}
    \left\| b_{\alpha i}\frac{\partial a_{ik}}{\partial x_{j}}\left(X\right)\frac{\partial u_{k}}{\partial y_{l}}\frac{\partial Y_{l}}{\partial x_{j}}\left(X\right) \right\|_{\rL^p(0,T;\rL^q(\Omega_\rf))} \le C R^2.
\end{equation*}

With regard to the subsequent addend of the form $\frac{\partial^{2}u_{\alpha}}{\partial y_{l}\partial y_{m}}(\frac{\partial Y_{l}}{\partial x_{j}}\left(X\right)\frac{\partial Y_{m}}{\partial x_{j}}\left(X\right)-\delta_{l,j}\delta_{m,j})$, we exploit \autoref{lem:est of terms related to the trafo}(a) and~(c) to get
\begin{equation}\label{eq:diff of Y terms}
    \begin{aligned}
        &\quad \left\| \frac{\partial Y_{l}}{\partial x_{j}}\left(X\right)\frac{\partial Y_{m}}{\partial x_{j}}\left(X\right)-\delta_{l,j}\delta_{m,j} \right\|_{\rL^\infty(0,T;\rL^\infty(\Omega_\rf))}\\
        &\le \left\| \left(\frac{\del Y_l}{\del x_j} - \delta_{lj}\right) \frac{\del Y_m}{\del x_j}\right\|_{\rL^\infty(0,T;\rL^\infty(\Omega_\rf))} + \left\| \delta_{lj} \left(\frac{\del Y_m}{\del x_j} - \delta_{mj}\right)\right\|_{\rL^\infty(0,T;\rL^\infty(\Omega_\rf))} \le C R.
    \end{aligned}
\end{equation}
Thus, by H\"older's inequality and a direct estimate of $\frac{\partial^{2}u_{\alpha}}{\partial y_{l}\partial y_{m}}$ by $\| u \|_{\E_1^u}$, we infer that
\begin{equation*}
    \left\| \frac{\partial^{2}u_{\alpha}}{\partial y_{l}\partial y_{m}}\left(\frac{\partial Y_{l}}{\partial x_{j}}\left(X\right)\frac{\partial Y_{m}}{\partial x_{j}}\left(X\right)-\delta_{l,j}\delta_{m,j}\right) \right\|_{\rL^p(0,T;\rL^q(\Omega_\rf))} \le C R^2.
\end{equation*}

In a similar manner, employing \autoref{lem:ests of plate and fluid vars}, \autoref{lem:est of terms related to the trafo}, \eqref{eq:est of der of a} and \eqref{eq:diff of Y terms}, we find that
\begin{equation*}
    \left\| b_{\alpha i}\frac{\partial a_{ik}}{\partial x_{j}}\left(X\right)a_{jm}\left(X\right)u_{k}u_{m} \right\|_{\rL^p(0,T;\rL^q(\Omega_\rf))} \le C R^3, \tand\\
\end{equation*}
\begin{equation*}
    \begin{aligned}
        &\left\| \frac{\partial u_{\alpha}}{\partial y_{l}}\frac{\partial^{2}Y_{l}}{\partial x_{j}^{2}}\left(X\right) \right\|_{\rL^p(0,T;\rL^q(\Omega_\rf))} + 
        \left\| \frac{\partial\pi}{\partial y_{k}}\det\left(\nabla X\right)\left(\frac{\partial Y_{\alpha}}{\partial x_{i}}\left(X\right)\frac{\partial Y_{k}}{\partial x_{i}}\left(X\right)-\delta_{\alpha,i}\delta_{k,i}\right) \right\|_{\rL^p(0,T;\rL^q(\Omega_\rf))}\\
        &+\left\| \frac{1}{\det\left(\nabla X\right)}\left[\left(u\cdot\nabla\right)u\right]_{\alpha} \right\|_{\rL^p(0,T;\rL^q(\Omega_\rf))} + \left\| \left[b\left(\partial_{t}a\right)\left(X\right)u\right]_{\alpha} \right\|_{\rL^p(0,T;\rL^q(\Omega_\rf))}\\
        &+ \left\| \left[\left(\nabla u\right)\left(\partial_{t}Y\right)\left(X\right)\right]_{\alpha} \right\|_{\rL^p(0,T;\rL^q(\Omega_\rf))} 
        \le C R^2.
    \end{aligned}
\end{equation*}
Concatenating the previous estimates, we deduce that $\| F(u,\pi,\eta_1,\eta_2) \|_{\E_0^u} \le C R^2$ for $C > 0$.

The second part of the proof is dedicated to the estimate of the other nonlinear term $G(u,\pi,\eta_1,\eta_2)$.
Here, for the first term $\partial_{s_{i}}\eta_1[\frac{\partial a_{ik}}{\partial x_{3}}\left(X\right)+\frac{\partial a_{3k}}{\partial x_{i}}\left(X\right)] u_k$, we employ \autoref{lem:ests of plate and fluid vars}(a) to estimate the first factor in $\rL^{\infty}(0,T;\rL^{\infty}(\omega))$ by $ C R$, estimate the factor in the middle by $C R$ upon observing that \eqref{eq:est of der of a} also holds when the norm is replaced by $\| \cdot \|_{\rL^{3p}(0,T;\rL^{3q}(\omega))}$, since the terms appearing in the estimate in \autoref{lem:est of terms related to the trafo}(d) are functions defined on $\omega$, and use \autoref{lem:ests of plate and fluid vars}(e) to estimate $u_k$, which is defined on a part of the fluid boundary, in $\rL^{\nicefrac{3p}{2}}(0,T;\rL^{\nicefrac{3q}{2}}(\del \Omega_\rf))$ by $C R$.
Thus, from H\"older's inequality, it follows that
\begin{equation*}
    \left\| \partial_{s_{i}}\eta_1\left[\frac{\partial a_{ik}}{\partial x_{3}}\left(X\right)+\frac{\partial a_{3k}}{\partial x_{i}}\left(X\right)\right] u_k \right\|_{\rL^p(0,T;\rL^q(\omega))} \le C R^3.
\end{equation*}
Likewise, we get $\| 2\frac{\partial a_{3k}}{\partial x_{3}}\left(X\right)u_{k} \|_{\rL^p(0,T;\rL^q(\omega))} \le C R^2$.
For $\partial_{s_{i}}\eta [a_{ik}\left(X\right)\frac{\partial Y_{l}}{\partial x_{3}}\left(X\right)+a_{3k}\left(X\right)\frac{\partial Y_{l}}{\partial x_{i}}\left(X\right)] \frac{\partial u_{k}}{\partial y_{l}}$, we use the same argument as before to estimate the first factor by $C R$, and we employ \autoref{lem:est of terms related to the trafo}(c) to bound the sum in the middle.
By virtue of \autoref{lem:ests of plate and fluid vars}(e), which enables an estimate of $\nabla u$ in $\rL^p(0,T;\rL^q(\del \Omega_\rf))$ by $C R$, we derive from H\"older's inequality that
\begin{equation*}
    \left\| \partial_{s_{i}}\eta\left[a_{ik}\left(X\right)\frac{\partial Y_{l}}{\partial x_{3}}\left(X\right)+a_{3k}\left(X\right)\frac{\partial Y_{l}}{\partial x_{i}}\left(X\right)\right] \frac{\partial u_{k}}{\partial y_{l}} \right\|_{\rL^p(0,T;\rL^q(\omega))} \le C R^2.
\end{equation*}
For the last term, we use \autoref{lem:est of terms related to the trafo}(a) and~(c) to get $\bigl\| a_{3k}(X)\frac{\partial Y_{l}}{\partial x_{3}}(X)-\delta_{3,k}\delta_{3,l}\bigr\|_{\rL^\infty(0,T;\rL^\infty(\omega))} \le CR$.
Thus, by the same estimate of $\nabla u$ as for the former term, joint with H\"older's inequality, it follows that
\begin{equation*}
    \left\| \left(a_{3k}\left(X\right)\frac{\partial Y_{l}}{\partial x_{3}}\left(X\right)-\delta_{3,k}\delta_{3,l}\right)\frac{\partial u_{k}}{\partial y_{l}} \right\|_{\rL^p(0,T;\rL^q(\omega))} \le C R^2.
\end{equation*}
Putting together the estimates of the four terms, we conclude that $G(u,\pi,\eta_1,\eta_2)$ can be estimated in~$\E_0^\eta = \rL^p(0,T;\rL^q(\omega))$ by $C R^2$, completing the proof of the proposition.
\end{proof}

The following proposition is used to establish the contraction property of the solution map to the linearized problem.
It can be proved in the same way as \autoref{prop:self-map ests of F & G}.
We omit the details here and refer to \cite[Prop.~5.3]{MT:21} for a rough outline of the proof.

\begin{prop}\label{prop:contr ests of F & G}
Consider $p$, $q \in (1,\infty)$ such that $\frac{1}{p} + \frac{3}{2q} < \frac{3}{2}$ as well as $v_1$, $v_2 \in \oB_{\E_1}(0,R)$, where $v_i = (u_i,\pi_i,\eta_{1,i},\eta_{2,i},d_{1,i},d_{2,i})$.
Then there is $C > 0$ such that for $F$ and $G$ from \eqref{eq:RHS F} and \eqref{eq:RHS G}, respectively, we have $\| F(v_1) - F(v_2) \|_{\E_0^u} + \| G(v_1) - G(v_2) \|_{\E_0^\eta} \le C R \cdot \| v_1 - v_2 \|_{\E_1}$.
\end{prop}

\subsection{Proof of existence of the strong time-periodic solution}
\ 

Prior to proving existence of a strong time-periodic solution to the multilayered fluid-structure interaction problem, let us make precise the fixed point argument.
For $\oB_{\E_1}(0,R)$ as above, we define
\begin{equation}\label{eq:sol map Phi}
    \Phi \colon \oB_{\E_1}(0,R) \to \E_1
\end{equation}
as follows:
Given $v = (u,\pi,\eta_1,\eta_2,d_1,d_2) \in \oB_{\E_1}(0,R)$, we denote by $\Tilde{v} \coloneqq \Phi(v)$ the solution to
\begin{equation}\label{eq:lin probl for fixed point arg}
    \left\{
    \begin{aligned}
        \del_t \Tilde{u} - \mdiv \sigma_\rf(\Tilde{u},\Tilde{\pi})
        &= F(u,\pi,\eta_1,\eta_2) + f, \, \mdiv \Tilde{u} = 0, &&\tin (0,T) \times \Omega_\rf,\\
        \del_t \Tilde{\eta}_1 = \Tilde{\eta}_2, \, \del_t \Tilde{\eta}_2 + P_\rm \Delta_s^2 \Tilde{\eta}_1 - \Delta_s \Tilde{\eta}_2
        &= P_\rm \left(G(u,\pi,\eta_1,\eta_2)\right) + P_\rm g\\
        &\qquad -P_\rm \Bigl(\sigma_\rf(\Tilde{u},\Tilde{\pi}) \bigr|_{\Gamma_0} \mre_3 \cdot \mre_3\Bigr)\\
        &\qquad + P_\rm \Bigl((\sigma_\rs(\Tilde{d}_1,\Tilde{d}_2) \mre_3)\bigr|_{\Gamma_0} \cdot \mre_3\Bigr), &&\tin (0,T) \times \omega,\\
        \int_\omega \Tilde{\eta}_1(t,s) \srd s 
        &= 0, &&\tfor t \in (0,T),\\
        \del_t \Tilde{d}_1 = \Tilde{d}_2, \, \del_t \Tilde{d}_2 - \mdiv \sigma_\rs(\Tilde{d}_1,\Tilde{d}_2)
        &= h, &&\tin (0,T) \times \Omega_\rs,\\
        \Tilde{u}
        &= P_\rm(\Tilde{\eta}_2) \mre_3, &&\ton (0,T) \times \Gamma_0,\\
        \Tilde{d}_1
        &= P_\rm(\Tilde{\eta}_1) \mre_3, &&\ton (0,T) \times \Gamma_0,\\
        \Tilde{u}
        &= 0, &&\ton (0,T) \times \Gamma_\rf,\\
        \Tilde{d}_1
        &= 0, &&\ton (0,T) \times \Gamma_\rs,\\
        \Tilde{\eta}_1 = \nabla_s \Tilde{\eta}_1 \cdot n_{\omega}
        &= 0, &&\ton (0,T) \times \del \omega,
    \end{aligned}
    \right.
\end{equation}
with periodicity conditions
\begin{equation*}
    \left\{
    \begin{aligned}
        \Tilde{u}(0,\cdot) 
        &= \Tilde{u}(T,\cdot), &&\tin \Omega_\rf,\\
        \Tilde{\eta}_1(0,\cdot) = \Tilde{\eta}_1(T,\cdot), \, \Tilde{\eta}_2(0,\cdot) 
        &= \Tilde{\eta}_2(T,\cdot), &&\tin \omega,\\
        \Tilde{d}_1(0,\cdot) = \Tilde{d}_1(T,\cdot), \, \Tilde{d}_2(0,\cdot) 
        &= \Tilde{d}_2(T,\cdot), &&\tin \Omega_\rs.
    \end{aligned}
    \right.
\end{equation*}
Note that the terms on the right-hand side $F$ and $G$ have been introduced in \eqref{eq:RHS F} and \eqref{eq:RHS G}, respectively.
Thanks to \autoref{cor:max per reg result multilayered}, the solution map is well-defined provided 
\begin{equation*}
    F(u,\pi,\eta_1,\eta_2) + f \in \rL^p(0,T;\rL^q(\Omega_\rf)^3), \enspace G(u,\pi,\eta_1,\eta_2) + g \in \rL^p(0,T;\rL^q(\omega)) \tand h \in \rL^p(0,T;\rL^q(\Omega_\rs)^3).
\end{equation*}
For the nonlinear terms $F$ and $G$, this is a result of \autoref{prop:self-map ests of F & G}, while for the external forcing terms, this property follows by assumption, see the assertions of \autoref{thm:ex of strong per sol} and \autoref{thm:strong time per sol in ref config}.

We now prove the existence of a strong time-periodic solution in the reference configuration.

\begin{proof}[Proof of \autoref{thm:strong time per sol in ref config}]
The aim is to show that for $R > 0$ sufficiently small, the solution map $\Phi$, introduced in \eqref{eq:sol map Phi} and representing the solution to the linearized problem \eqref{eq:lin probl for fixed point arg}, is a self-map and contraction on the ball~$\oB_{\E_1}(0,R)$ for $R > 0$ sufficiently small and sufficiently small external forcing terms $f$, $g$ and $h$.

For this purpose, let $v = (u,\pi,\eta_1,\eta_2,d_1,d_2) \in \oB_{\E_1}(0,R)$.
From \autoref{cor:max per reg result multilayered}, it follows that $\Phi(v) \in \E_1$.
Additionally invoking \autoref{prop:self-map ests of F & G} and the assumption on the external forcing terms, we also get
\begin{equation*}
    \begin{aligned}
        \| \Phi(v) \|_{\E_1} \le C\bigl(\| F(u,\pi,\eta_1,\eta_2) \|_{\E_0^u} + \| f \|_{\E_0^u} + \| G(u,\pi,\eta_1,\eta_2) \|_{\E_0^\eta} + \| g \|_{\E_0^\eta} + \| h \|_{\E_0^d}\bigr) \le C(R^2 + \eps).
    \end{aligned}
\end{equation*}
On the other hand, for $v_1$, $v_2 \in \oB_{\E_1}(0,R)$, by \autoref{prop:contr ests of F & G}, we find
\begin{equation*}
    \| \Phi(v_1) - \Phi(v_2) \|_{\E_1} \le C R \cdot \| v_1 - v_2 \|_{\E_1}.
\end{equation*}
Moreover, by the embedding from \autoref{lem:embs of the max reg space}(b) and the Sobolev embedding $\rB_{qp}^{4-\nicefrac{2}{p}}(\omega) \hookrightarrow \rL^\infty(\omega)$, which holds for all $p$, $q \in (1,\infty)$, see, e.g., \cite[Thm.~4.6.1]{Tri:78}, for $v \in \oB_{\E_1}(0,R)$, there exists a constant $C_1 > 0$ such that
\begin{equation}\label{eq:est of eta}
    \| \eta(t,\cdot) \|_{\rL^\infty(\omega)} \le C_1 \cdot \| v \|_{\E_1} \le C_1 R \tforall t \in [0,T].
\end{equation}

Thus, choosing $R \coloneqq \min\left\{\frac{1}{2},\frac{1}{4C},\frac{\delta_0}{C_1}\right\}$, where $\delta_0$ is as introduced in \eqref{eq:concr bound for plate displacement}, and $\eps = \eps(R) \coloneqq \frac{R}{4C}$, we get
\begin{equation*}
    \| \Phi(v) \|_{\E_1} \le \frac{R}{2} \tand \| \Phi(v_1) - \Phi(v_2) \|_{\E_1} \le \frac{1}{2} \cdot \| v_1 - v_2 \|_{\E_1}.
\end{equation*}
In other words, $\Phi$ is a self-map and contraction on $\oB_{\E_1}(0,R)$, and the contraction mapping principle yields the existence of a unique fixed point of $\Phi$.
By construction, this shows the existence of a unique solution $v = (u,\pi,\eta_1,\eta_2,d_1,d_2) \in \oB_{\E_1}(0,R)$ to the transformed interaction problem \eqref{eq:transformed syst}.
Let us observe that the regularities asserted in \autoref{thm:strong time per sol in ref config} are a consequence of the definition of $\E_1$.
In addition, thanks to $R \le \frac{\delta_0}{C_1}$ and \eqref{eq:est of eta}, it follows that the smallness condition \eqref{eq:smallness ass displacement} is satisfied for all $t \in [0,T]$.
This yields the diffeomorphism property of $X(t,\cdot)$ for all $t \in [0,T]$, see also \autoref{rem:smallness of eta and diffeo}, and it also implied that no self-intersection of the domain occurs, i.e., the domain does not degenerate.
\end{proof}

The proof of \autoref{thm:ex of strong per sol} now is a direct consequence of \autoref{thm:strong time per sol in ref config} that we have just proved.

\medskip 

{\bf Acknowledgements.}
\small{The authors would like to thank the anonymous referees for their careful reading of the manuscript and their valuable comments.
This research is supported by the Basque Government through the BERC 2022-2025 program and by the Spanish State Research Agency through BCAM Severo Ochoa excellence accreditation CEX2021-01142-S funded by MICIU/AEI/10.13039/501100011033 and through Grant PID2023-146764NB-I00 funded by MICIU/AEI/10.13039/501100011033 and cofunded by the European Union. 
Felix Brandt acknowledges the support by the German National Academy of Sciences Leopoldina through the Leopoldina Fellowship Program with grant number~LPDS 2024-07 and the support by DFG project FOR~5528. 
He also extends his gratitude to the Basque Center for Applied Mathematics for their generous hospitality during his visit, where a portion of this work was completed.
Arnab Roy would like to thank the Alexander von Humboldt-Foundation, and the support by Grant RYC2022-036183-I funded by MICIU/AEI/10.13039/501100011033 and by ESF+.}

\bibliography{strong_multilayered_FSI}
\bibliographystyle{siam}

\end{document}